\documentclass[reqno]{amsart}
\usepackage{amssymb,enumerate,mathrsfs,curves}

\numberwithin{equation}{section}

\newtheorem{theorem}[equation]{Theorem}
\newtheorem{proposition}[equation]{Proposition}
\newtheorem{lemma}[equation]{Lemma}
\newtheorem{corollary}[equation]{Corollary}

\theoremstyle{definition}
\newtheorem{definition}[equation]{Definition}
\newtheorem{remark}[equation]{Remark}
\newtheorem{example}[equation]{Example}

\def\C{\mathbb C}
\def\Dom{\mathcal{D}}

\def\E{\mathfrak E}
\def\F{\mathcal F}
\def\K{\mathcal K}

\def\L{\mathscr L}
\def\M{\mathfrak M}
\def\N{\mathbb N}
\def\R{\mathbb R}
\def\S{\mathscr S}
\def\Sing{\mathcal E}

\def\Oh{\mathcal O}
\def\Vee{\mathcal V}
\def\W{\mathcal W}
\def\Z{\mathbb Z}

\def\Strip{S}
\def\smallset#1{\{#1\}}

\def\a{\mathfrak a}
\def\im{i}
\def\ie{i.e.}
\def\id{I}
\def\st{\text{ s.t. }}

\DeclareMathOperator{\ord}{ord}
\DeclareMathOperator{\coeff}{c}
\DeclareMathOperator{\resolv}{res}

\def\bkappa{\pmb \kappa}

\def\csym{\,{}^c\!\sym}

\def\e{\mathrm e}

\def\eps{\varepsilon}
\def\m{\mathfrak m}
\def\minus{\backslash}
\def\norm#1{\left\|{#1}\right\|}
\def\open#1{\smash[t]{\overset{{}_{\circ}}{#1}{}}}
\def\r{\mathscr R}
\def\set#1{\left\{#1\right\}}

\def\vp{\varphi}

\def\Gr{\mathrm{Gr}}

\DeclareMathOperator{\LinSpan}{span}
\DeclareMathOperator{\Aut}{Aut}
\DeclareMathOperator{\bgres}{bg-res}

\DeclareMathOperator{\dist}{dist}
\DeclareMathOperator{\End}{End}
\DeclareMathOperator{\Ind}{ind}
\DeclareMathOperator{\hol}{hol}
\DeclareMathOperator{\Hom}{Hom}

\DeclareMathOperator{\Tr}{Tr}
\DeclareMathOperator{\Diff}{Diff}
\DeclareMathOperator{\spec}{spec}
\DeclareMathOperator{\sym}{ \sigma\!\!\!\sigma}

\begin{document}
\title{Dynamics on Grassmannians and resolvents of cone operators}
\author{Juan B. Gil}
\author{Thomas Krainer}
\address{Penn State Altoona\\ 3000 Ivyside Park \\ Altoona, PA 16601-3760}
\author{Gerardo A. Mendoza}
\address{Department of Mathematics\\ Temple University\\ 
Philadelphia, PA 19122}
\begin{abstract}
The paper proves the existence and elucidates the structure of the asymptotic expansion of the trace of the resolvent of a closed extension of a general elliptic cone operator on a compact manifold with boundary as the spectral parameter tends to infinity. The hypotheses involve only minimal conditions on the symbols of the operator. The results combine previous investigations by the authors on the subject with an analysis of the asymptotics of a family of projections related to the domain. This entails a fairly detailed study of the dynamics of a flow on the Grassmannian of domains.
\end{abstract}

\maketitle

\section{Introduction}

In \cite{GKM5a} we analyzed the behavior of the trace of the resolvent of an elliptic cone operator on a compact manifold as the spectral parameter increases radially assuming, in addition to natural ray conditions on its symbols, that the domain is stationary. Here we complete our analysis with Theorem~\ref{ResolventTraceExpansion2}, which describes the behavior of the aforementioned trace without any restriction on the domain. The main new ingredient is Theorem~\ref{AsymptOfProjB} on the asymptotics of a family of projections related to the domain. This involves a fairly detailed analysis of the dynamics of a flow on the Grassmannian of domains.

Let $M$ be a smooth compact $n$-dimensional manifold with boundary $Y$. A cone operator on $M$ is an element $A\in x^{-m}\Diff^m_b(M;E)$, $m>0$; here $\Diff^m_b(M;E)$ is the space of $b$-differential operators of Melrose \cite{RBM2} acting on sections of a vector bundle $E\to M$ and $x$ is a defining function of $Y$ in $M$, positive in $\open M$. Associated with such an operator is a pair of symbols, the $c$-symbol $\csym(A)$ and the wedge symbol $A_\wedge$. The former is a bundle endomorphism closely related to the regular principal symbol of $A$, indeed ellipticity is defined as the invertibility of $\csym(A)$. The wedge symbol is a partial differential operator on $N_+Y$, the closed inward pointing normal bundle of $Y$ in $M$, essentially the original operator with coefficients frozen at the boundary. See \cite[Section 2]{GKM5a} for a brief overview and \cite[Section 3]{GKM1} for a detailed exposition of basic facts concerning cone operators.

Fix a Hermitian metric on $E$ and a smooth positive $b$-density $\m_b$ on $M$ ($x\m_b$ is a smooth everywhere positive density on $M$) to define the spaces $x^\gamma L^2_b(M;E)$. Let $A$ be a cone operator. The unbounded operator
\begin{equation}\label{OrgOp}
A:C_c^\infty(\open M;E)\subset x^\gamma L^2_b(M;E)\to x^\gamma L^2_b(M;E)
\end{equation}
admits a variety of closed extensions with domains $\Dom\subset x^\gamma L^2_b(M;E)$ such that $\Dom_{\min}\subset \Dom\subset \Dom_{\max}$, where $\Dom_{\min}$ is the domain of the closure of \eqref{OrgOp} and 
\begin{equation*}
\Dom_{\max}=\set{u\in x^\gamma L^2_b(M;E):Au\in x^\gamma L^2_b(M;E)}.
\end{equation*}
When $A$ is $c$-elliptic,  $A$ is Fredholm with any such domain (Proposition 1.3.16 of Lesch \cite{Le97}). The set of closed extensions is parametrized by the elements of the various Grassmannian manifolds associated with the finite-dimensional space $\Dom_{\max}/\Dom_{\min}$, a useful point of view exploited extensively in \cite{GKM1}. Without loss of generality, assume $\gamma=-m/2$. 

Associated with $N_+Y$ there are analogous Hilbert spaces $x_\wedge^{-m/2} L^2_b(N_+Y;E_\wedge)$. Here $x_\wedge$ is the function determined by $dx$ on $N_+Y$, $E_\wedge$ is the pull back of $E|_Y$ to $N_+Y$, and the density is $x_\wedge^{-1}\m_Y$ where $\m_Y$ is the density on $Y$ obtained by contraction of $\m_b$ with $x\partial_x$. We will drop the subscript $\wedge$ from $x_\wedge$ and $E_\wedge$, and trivialize $N_+Y$ as $Y^\wedge=[0,\infty)\times Y$ using the defining function. The space $x^{-m/2} L^2_b(Y^\wedge;E)$ carries a natural unitary $\R_+$ action $(\varrho,u)\mapsto \kappa_\varrho u$ which after fixing a Hermitian connection on $E$ is given by
\begin{equation*}
\kappa_\varrho u(x,y)=\varrho^{m/2}u(\varrho x,y)\; \text{ for } \varrho>0,\ (x,y)\in Y^\wedge.
\end{equation*}

The minimal and maximal domains, $\Dom_{\wedge,\min}$ and $\Dom_{\wedge,\max}$, of $A_\wedge$ are defined in an analogous fashion as those of $A$, the first of these spaces being the domain of the closure of
\begin{equation}\label{OrgOpWedge}
A_\wedge:C_c^\infty(\open Y^\wedge;E)\subset x^{-m/2} L^2_b(Y^\wedge;E)
\to x^{-m/2} L^2_b(Y^\wedge;E).
\end{equation}
A fundamental property of $A_\wedge$ is its $\kappa$-homogeneity, $\kappa_\varrho A_\wedge = \varrho^{-m} A_\wedge\kappa_\varrho$. Thus $\Dom_{\wedge,\min}$ and $\Dom_{\wedge,\max}$ are both $\kappa$-invariant, hence there is an $\R_+$ action 
\begin{equation*}
 \varrho\mapsto
 \pmb \kappa_{\varrho}:\Dom_{\wedge,\max}/\Dom_{\wedge,\min}\to \Dom_{\wedge,\max}/\Dom_{\wedge,\min}
\end{equation*}
which in turn induces for each $d''$ an action on $\Gr_{d''}(\Dom_{\wedge,\max}/\Dom_{\wedge,\min})$, the Grassmannian of $d''$-dimensional subspaces of $\Dom_{\wedge,\max}/\Dom_{\wedge,\min}$. Observe that since the quotient is finite dimensional these actions extend holomorphically to $\C\minus \overline{\R}_{-}$.

Assuming the $c$-ellipticity of $A$ we constructed in \cite[Theorem 4.7]{GKM1} and reviewed in \cite[Section 2]{GKM5a} a natural isomorphism
\begin{equation*}
\theta: \Dom_{\max}/\Dom_{\min} \to \Dom_{\wedge,\max}/\Dom_{\wedge,\min}
\end{equation*}
allowing in particular passage from a domain $\Dom$ for $A$ to a domain $\Dom_{\wedge}$ for $A_\wedge$ which we shall call the associated domain.

We showed in \cite{GKM2} that if
\begin{equation}\label{RayConditions}
\mbox{
\parbox{100mm} {
$\csym(A)-\lambda$ is invertible for $\lambda$ in a closed sector
$\Lambda\subsetneq \C$ which is a sector of minimal growth for $A_\wedge$ with
the associated domain $\Dom_\wedge$ defined via
$\Dom_\wedge/\Dom_{\wedge,\min}=\theta(\Dom/\Dom_{\min})$,
}}
\end{equation}
then $\Lambda$ is also a sector of minimal growth for $A_\Dom$, the operator $A$ with domain $\Dom$, and for $\ell\in\N$ sufficiently large, $(A_{\Dom}-\lambda)^{-\ell}$ is an analytic family of trace class operators. In \cite{GKM5a} we gave the asymptotic expansion of
$\Tr(A_\Dom-\lambda)^{-\ell}$ under the condition that $\Dom$ was stationary. Recall that a subspace $\Dom\subset \Dom_{\max}$ with $\Dom_{\min}\subset\Dom$ is said to be stationary if $\theta(\Dom/\Dom_{\wedge,\max})\in\Gr_{d''}(\Dom_{\wedge,\max}/\Dom_{\wedge,\min})$ is a fixed point of the action $\bkappa$. More generally, assuming only \eqref{RayConditions}, we now prove:

\begin{theorem}\label{ResolventTraceExpansion2}
For any $\varphi \in C^{\infty}(M;\textup{End}(E))$ and $\ell\in\N$ with $m\ell>n$,
\begin{equation*}
\Tr \bigl(\varphi(A_{\Dom}-\lambda)^{-\ell}\bigr) \sim \sum_{j=0}^{\infty}
r_j(\lambda^{i\mu_{1}},\ldots,\lambda^{i\mu_{N}},\log\lambda) \lambda^{\nu_j/m} 
\; \text{ as } |\lambda| \to \infty,
\end{equation*}
where each $r_j$ is a rational function in $N+1$ variables, $N \in \N_0$, with real numbers $\mu_{k}$, $k = 1,\ldots,N$, and $\nu_j > \nu_{j+1} \to -\infty$ as
$j \to \infty$. We have $r_j = p_j/q_j$ with $p_j,\; q_j \in \C[z_1,\ldots,z_{N+1}]$ such that $q_j(\lambda^{i\mu_{1}},\ldots,\lambda^{i\mu_{N}},\log\lambda)$ is
uniformly bounded away from zero for large $\lambda$.
\end{theorem}

The above expansion is to be understood as the asymptotic expansion of a symbol into its components as discussed in the appendix. As shown in \cite{GKM5a}, 
\begin{equation*}
\Tr \bigl(\varphi(A_{\Dom}-\lambda)^{-\ell}\bigr) \sim
\sum\limits_{j=0}^{n-1}\alpha_{j}\lambda^{\frac{n-\ell m-j}{m}}
+\alpha_{n}\log(\lambda) \lambda^{-\ell} + s_{\Dom}(\lambda)
\end{equation*}
with coefficients $\alpha_j\in \C$ that are independent of the choice of domain $\Dom$, and a remainder $s_{\Dom}(\lambda)$ of order $\Oh(|\lambda|^{-\ell})$. Here we will show that $s_{\Dom}(\lambda)$ is in fact a symbol that admits an expansion into components that exhibit in general the structure shown in Theorem~\ref{ResolventTraceExpansion2}. More precisely, let 
\begin{equation}\label{PhasesMainThm}
\M = \set{\Re \sigma/m : \sigma \in \spec_b(A), \; -m/2 < \Im \sigma < m/2},
\end{equation}
where $\spec_b(A)$ denotes the boundary spectrum of $A$ (see \cite{RBM2}), and let
\begin{multline}\label{SemigroupDef}
\E = \text{additive semigroup generated by}\\
\{\Im(\sigma - \sigma') : \sigma,\,\sigma' \in \spec_b(A),\,
-m/2 < \Im\sigma\leq\Im\sigma' < m/2\}\cup(-\N_0), \;\; 
\end{multline}
a discrete subset of $\overline \R_-$ without points of accumulation. Then 
\begin{equation}\label{sDExpansion}
s_{\Dom}(\lambda)\sim \sum_{\substack{\nu\in \E\\\nu\leq -\ell m}}
r_\nu(\lambda^{i\mu_{1}},\ldots,\lambda^{i\mu_{N}},\log\lambda) \lambda^{\nu/m} 
\; \text{ as } |\lambda| \to \infty,
\end{equation}
where the $\mu_i$ are the elements of $\M$ and the $r_\nu$ are rational functions of their arguments as described in the theorem.

An analysis of the arguments of Sections~\ref{sec:LimitingOrbits} and \ref{sec:AsymptoticsOfProjection} shows that the structure of the functions $r_\nu$ depends strongly on the relation of the domain with the part of the boundary spectrum in the `critical strip' $\set{\sigma \in \C : -m/2 < \Im\sigma < m/2}$. This includes what elements of the set $\M$ actually appear in the $r_\nu$, and whether they are truly rational functions and not just polynomials. We will not follow up on this observation in detail, but only single out here the following two cases because of their special role in the existing literature. When $\Dom$ is stationary, the machinery of Sections~\ref{sec:LimitingOrbits} and \ref{sec:AsymptoticsOfProjection} is not needed, and we recover the results of \cite{GKM5a}: the $r_\nu$ are just polynomials in $\log\lambda$, and the numbers $\nu$ in \eqref{sDExpansion} are all integers. If $\Dom$ is nonstationary, but the elements of $\spec_b(A)$ in the critical strip are vertically aligned, then again there is no dependence on the elements of $\M$, but the coefficients are generically rational functions of $\log\lambda$. Note that all second order regular singular operators in the sense of Br{\"u}ning and Seeley (see \cite{BrSe87,BrSe91,KLP08b}) have this special property.

By standard arguments, Theorem~\ref{ResolventTraceExpansion2} implies corresponding results about the expansion of the heat trace $\Tr\bigl(\varphi e^{-tA_{\Dom}}\bigr)$ as $t \to 0^+$ if $A_{\Dom}$ is sectorial, and about the structure of the $\zeta$-function if $A_{\Dom}$ is positive. It has been observed by other authors that the resolvent trace, the heat kernel, and the $\zeta$-function for certain model operators may exhibit so called \emph{unusual} or \emph{exotic} behavior \cite{FMP04,FMPSeeley,FPW,KLP06,KLP08a,KLP08b,LMP07}. This is accounted for in Theorem~\ref{ResolventTraceExpansion2} by the fact that the components may have non-integer orders $\nu_j$ belonging to the set $\E$, and that the $r_j$ may be genuine rational functions and not mere polynomials. For example, the former implies that the $\zeta$-function of a positive operator might have poles at unusual locations, and the latter that it might not extend meromorphically to $\C$ at all. Both phenomena have been observed for $\zeta$-functions of model operators.

Earlier investigations on this subject typically relied on separation of variables and special function techniques to carry out the analysis near the boundary. This is one major reason why all previously known results are limited to narrow classes of operators. Here and in \cite{GKM5a} we develop a new approach which leads to the completely general result Theorem~\ref{ResolventTraceExpansion2}. This result is new even for Laplacians with respect to warped cone metrics, or, more generally, for $c$-Laplacians (see \cite{GKM5a}).

Throughout this paper we assume that the ray conditions \eqref{RayConditions} hold. We will rely heavily on \cite{GKM5a}, where we analyzed $(A_\Dom-\lambda)^{-\ell}$ using the representation
\begin{equation}\label{ResolventStructure}
 (A_\Dom-\lambda)^{-1}=B(\lambda)+\big[1-B(\lambda)(A-\lambda)\big] F_\Dom(\lambda)^{-1} T(\lambda)
\end{equation}
obtained in \cite{GKM2} with the aid of the formula
\begin{equation*}
(A_\Dom-\lambda)^{-\ell}= \frac{1}{(\ell-1)!}\partial_{\lambda}^{\ell-1}(A_\Dom-\lambda)^{-1}.
\end{equation*}
In \cite{GKM5a} we described in full generality the asymptotic behavior of the operator families $B(\lambda)$, $[1-B(\lambda)(A-\lambda)]$, and $T(\lambda)$, and gave an asymptotic expansion of $F_\Dom(\lambda)^{-1}$ if $\Dom$ is stationary. Therefore, to complete the picture we only need to show that $F_\Dom(\lambda)^{-1}$ has a full asymptotic expansion and describe its qualitative features for a general domain $\Dom$.

\medskip
We end this introduction with an overview of the paper. There is a formula similar to \eqref{ResolventStructure} concerning the extension of \eqref{OrgOpWedge} with domain $\Dom_\wedge$. The analysis of $F_\Dom(\lambda)^{-1}$ in \cite{GKM5a} was facilitated by the fact that the corresponding operator $F_{\wedge,\Dom_{\wedge}}(\lambda)^{-1}$ for $A_{\wedge,\Dom_{\wedge}}$ has a simple homogeneity property when $\Dom$ is stationary. In Section~\ref{sec:ModelResolvent} we will establish an explicit connection between the operator $F_{\wedge,\Dom_{\wedge}}(\lambda)^{-1}$ and a family of projections for a general domain $\Dom_{\wedge}$. This family of projections, previously studied in the context of rays of minimal growth in \cite{GKM1,GKM3}, is analyzed further in Sections~\ref{sec:LimitingOrbits} and \ref{sec:AsymptoticsOfProjection}, and is shown to fully determine the asymptotic structure of $F_{\wedge,\Dom_{\wedge}}(\lambda)^{-1}$, summarized in Proposition~\ref{FwedgeInvSymbol}. As a consequence, we obtain in Proposition~\ref{ResolventAwedge} a description of the asymptotic structure of $(A_{\wedge,\Dom_{\wedge}}-\lambda)^{-1}$.

The family of projections is closely related to the curve through $\Dom_{\wedge}/\Dom_{\wedge,\min}$ determined by the flow defined by $\bkappa$ on $\Gr_{d''}(\Dom_{\wedge,\max}/\Dom_{\wedge,\min})$. The behavior of an abstract version of $\bkappa^{-1}_{\zeta}(\Dom_{\wedge}/\Dom_{\wedge,\min})$ is analyzed in extenso in Section~\ref{sec:LimitingOrbits}. Let $\Sing$ denote a finite dimensional complex vector space and $\a:\Sing\to\Sing$ an arbitrary linear map. The main technical result of Section~\ref{sec:LimitingOrbits} is an algorithm (Lemmas~\ref{ExistenceOfShadowLemma} and \ref{BasicLemma}) which is used to obtain a section of the variety of frames of elements of $\Gr_{d''}(\Sing)$ along $e^{t\a}D$ for all sufficiently large $t$ (really, all complex $t$ with $|\Im t|\leq\theta$ and $\Re t$ large). The dependence of the section on $t$ is explicit enough to allow the determination of the nature of the $\Omega$-limit sets of the flow $t\mapsto e^{t\a}$ on $\Gr_{d''}(\Sing)$ (Proposition~\ref{ExistenceOfShadow}).

The results of Section~\ref{sec:LimitingOrbits} are used in Section~\ref{sec:AsymptoticsOfProjection} to obtain the asymptotic behavior of the aforementioned family of projections, and consequently of $F_{\wedge,\Dom_\wedge}(\lambda)^{-1}$ when $\lambda\in \Lambda$ as $|\lambda|\to \infty$, assuming only the ray condition \eqref{RayConditions} for $A_\wedge$ on $\Dom_{\wedge}$ (in the equivalent form given by (iii) of Theorem~\ref{EquivRayConditionsWedge}).

The work comes together in Section~\ref{sec:AymtpoticsResolvent}. There we obtain first the full asymptotics of $F_{\Dom}(\lambda)^{-1}$ using results from \cite{GKM2,GKM5a} and the asymptotics of $F_{\wedge,\Dom_\wedge}(\lambda)^{-1}$ obtained earlier. This is then combined with work done in \cite{GKM5a} on the asymptotics of the rest of the operators in \eqref{ResolventStructure}, giving Theorem~\ref{MainThmTraceExpansion} on the asymptotics of the trace $\Tr\bigl(\varphi(A_{\Dom}-\lambda)^{-\ell}\bigr)$. The manipulation of symbols and their asymptotics is carried out within the framework of refined classes of symbols discussed in the appendix.

\section{Resolvent of the model operator}
\label{sec:ModelResolvent}

In \cite{GKM1,GKM2,GKM3} we studied the existence of sectors of minimal growth and the structure of resolvents for the closed extensions of an elliptic cone operator $A$ and its wedge symbol $A_\wedge$. In particular, in \cite{GKM2} we determined that $\Lambda$ is a sector of minimal growth for $A_\Dom$ if $\csym(A)-\lambda$ is invertible for $\lambda$ in $\Lambda$, and if $\Lambda$ is also a sector of minimal growth for $A_\wedge$ with the associated domain $\Dom_\wedge$. In this section we will briefly review and refine some of the results concerning the resolvent of $A_{\wedge,\Dom_\wedge}$ in the closed sector $\Lambda$. 

The set
\begin{equation*}
\bgres(A_\wedge) =\set{\lambda\in \C:A_\wedge-\lambda\text{ is injective on } \Dom_{\wedge,\min} \text{ and surjective on } \Dom_{\wedge,\max} },
\end{equation*}
introduced in \cite{GKM1}, is of interest for a number of reasons, including the property that if $\lambda\in \bgres(A_\wedge)$ then every closed extension of $A_\wedge-\lambda$ is Fredholm.
Using the property
\begin{equation}\label{HomogeneityOfAwedge}
\kappa_\varrho A_\wedge = \varrho^{-m} A_\wedge\kappa_\varrho
\end{equation}
one verifies that $\bgres(A_\wedge)$ is a disjoint union of open sectors in $\C$. Defining $d''=-\Ind(A_{\wedge,\min}-\lambda)$ and $d'=\Ind(A_{\wedge\max}-\lambda)$ for $\lambda$ in one of these sectors, one has that if $(A_{\wedge,\Dom_\wedge}-\lambda)$ is invertible, then $\dim(\Dom_\wedge/\Dom_{\wedge,\min})=d''$ and $\dim \ker(A_{\wedge,\max}-\lambda)=d'$. The dimension of $\Dom_{\wedge,\max}/\Dom_{\wedge,\min}$ is $d'+d''$.

From now on we assume that $\Lambda\ne \C$ is a fixed closed sector such that $\Lambda\minus 0\subset\bgres(A_\wedge)$ and $\resolv A_{\wedge,\Dom_{\wedge}}\cap \Lambda\ne \emptyset$. Without loss of generality we also assume that $\Lambda$ has nonempty interior. The set $\resolv A_{\wedge,\Dom_{\wedge}}\cap \Lambda$ is discrete, in particular connected.

Corresponding to \eqref{ResolventStructure} there is a representation
\begin{equation}\label{ResolventWedgeStructure}
 (A_{\wedge,\Dom_\wedge}-\lambda)^{-1}=B_\wedge (\lambda)+\big[1-B_\wedge (\lambda)(A_\wedge-\lambda)\big] F_{\wedge,\Dom_\wedge}(\lambda)^{-1} T_\wedge (\lambda)
\end{equation}
for $\lambda\in\Lambda\cap \resolv(A_{\wedge,\Dom_\wedge})$. As we shall see in Section~\ref{sec:AymtpoticsResolvent}, if $\Lambda$ is a sector of minimal growth for $A_{\wedge,\Dom_\wedge}$, then the asymptotic structure of $F_{\wedge,\Dom_\wedge}(\lambda)^{-1}$ determines much of the asymptotic structure of the operator $F_\Dom(\lambda)^{-1}$ in \eqref{ResolventStructure}. 

If $\Dom_{\wedge}$ is $\kappa$-invariant, then $F_{\wedge,\Dom_\wedge}(\lambda)^{-1}$ has the homogeneity property
\begin{equation}\label{FwedgeInvHomogeneity}
\bkappa_{|\lambda|^{1/m}}^{-1} F_{\wedge,\Dom_\wedge}(\lambda)^{-1} = F_{\wedge,\Dom_\wedge}(\hat\lambda)^{-1}
\end{equation}
and is, in that sense, the principal homogeneous component of $F_\Dom(\lambda)^{-1}$. This facilitates the expansion of $F_\Dom(\lambda)^{-1}$ as shown in \cite[Proposition~5.17]{GKM5a}. However, if $\Dom_\wedge$ is not $\kappa$-invariant, $F_{\wedge,\Dom_\wedge}(\lambda)^{-1}$ fails to be homogeneous and its asymptotic behavior is more intricate.

\medskip
The identity \eqref{ResolventWedgeStructure} obtained in \cite{GKM2} begins with a choice of a family of operators $K_\wedge(\lambda):\C^{d''}\to x^{-m/2}L^2_b(Y^\wedge;E)$ which is $\kappa$-homogeneous of degree $m$ and such that 
\begin{equation*}
\begin{pmatrix}
A_\wedge - \lambda & K_\wedge(\lambda)
\end{pmatrix} :\begin{array}{c}\Dom_{\wedge,\min}\\\oplus\\\C^{d''}\end{array}\to x^{-m/2}L^2_b(Y^\wedge;E)
\end{equation*}
is invertible for all $\lambda\in \Lambda\minus 0$. The homogeneity condition on $K_\wedge$ means that
\begin{equation}\label{homogeneityK}
K_\wedge(\varrho^m\lambda)=\varrho^m\kappa_\varrho K_\wedge(\lambda)
\;\text{ for } \varrho>0.
\end{equation}
Defining the action of $\R_+$ on $\C^{d''}$ to be the trivial action, this condition on the family $K_\wedge(\lambda)$ becomes the same homogeneity property that the family $A_\wedge-\lambda$ has because of \eqref{HomogeneityOfAwedge}. Other than this the choice of $K_\wedge$ is largely at our disposal. That such a family $K_\wedge(\lambda)$ exists is guaranteed by the condition that $\Lambda\minus 0\subset \bgres(A_\wedge)$.

Let $\lambda_0\in\open\Lambda$ be such that $A_{\wedge,\Dom_\wedge}-\lambda$ is invertible for every $\lambda=e^{i\vartheta}\lambda_0\in\Lambda$. We fix $\lambda_0$ (for convenience on the central axis of the sector) and a cut-off function $\omega\in C_c^\infty([0,1))$, and define
\begin{equation}\label{Kwedge}
 K_\wedge(\lambda)= (A_\wedge-\lambda)\omega(x|\lambda|^{1/m}) \bkappa_{|\lambda/\lambda_0|^{1/m}} \;\text{ for } \lambda\in\Lambda\minus 0
\end{equation}
acting on $\Dom_\wedge/\Dom_{\wedge,\min}\cong\C^{d''}$. The factor $\omega(x|\lambda|^{1/m}) \bkappa_{|\lambda/\lambda_0|^{1/m}}$ in \eqref{Kwedge} is to be understood as the composition 
\begin{equation*}
\Dom_\wedge/\Dom_{\wedge,\min}\xrightarrow{\bkappa_{|\lambda/\lambda_0|^{1/m}}}\Dom_{\wedge,\max}/\Dom_{\wedge,\min}\cong \Sing_{\wedge,\max}\subset \Dom_{\wedge,\max}\xrightarrow{\omega(x|\lambda|^{1/m})} \Dom_{\wedge,\max}
\end{equation*}
in which the last operator is multiplication by the function $\omega(x|\lambda|^{1/m})$ and we use the canonical identification of $\Dom_{\wedge,\max}/\Dom_{\wedge,\min}$ with the orthogonal complement $\Sing_{\wedge,\max}$ of $\Dom_{\wedge,\min}$ in $\Dom_{\wedge,\max}$ using the graph inner product,
\begin{equation*}
(u,v)_{A_\wedge} = (A_\wedge u,A_\wedge v)+(u,v),\quad u,v\in \Dom_{\wedge,\max}.
\end{equation*}
By definition, $K_\wedge(\lambda)$ satisfies \eqref{homogeneityK} and the family
\begin{equation*}
\begin{pmatrix}
A_\wedge - \lambda & K_\wedge(\lambda)
\end{pmatrix} :\begin{array}{c}\Dom_{\wedge,\min}\\ \oplus\\ \Dom_\wedge/\Dom_{\wedge,\min} \end{array}\to x^{-m/2}L^2_b(Y^\wedge;E)
\end{equation*}
is invertible for every $\lambda$ on the arc $\set{\lambda\in\Lambda:|\lambda|=|\lambda_0|}$ through $\lambda_0$. Therefore, using $\kappa$-homogeneity, it is invertible for every $\lambda\in\Lambda\minus 0$. If 
\begin{equation*}
\binom{B_\wedge(\lambda)}{T_\wedge(\lambda)}:x^{-m/2}L_b^2(Y^\wedge;E)\to \begin{array}{c}\Dom_{\wedge,\min}\\ \oplus \\ \Dom_\wedge/\Dom_{\wedge,\min}\end{array}
\end{equation*}
is the inverse of $\begin{pmatrix} A_{\wedge,\min}-\lambda & K_\wedge(\lambda) \end{pmatrix}$, then $T_\wedge(\lambda)(A_\wedge-\lambda)=0$ on $\Dom_{\wedge,\min}$, so it induces a map
\begin{equation*}
 F_\wedge(\lambda)=[T_\wedge(\lambda)(A_\wedge-\lambda)]: \Dom_{\wedge,\max}/\Dom_{\wedge,\min}\to \Dom_\wedge/\Dom_{\wedge,\min}
\end{equation*}
whose restriction $F_{\wedge,\Dom_\wedge}(\lambda)= F_\wedge(\lambda)|_{\Dom_\wedge/\Dom_{\wedge,\min}}$ is invertible for $\lambda \in \resolv(A_{\wedge,\Dom_{\wedge}})\cap \Lambda\minus 0$ and leads to \eqref{ResolventWedgeStructure}. Moreover, since $T_\wedge(\lambda)K_\wedge(\lambda)=1$, we have
\begin{align*}
F_{\wedge,\Dom_\wedge}(\lambda)^{-1} &= 
q_\wedge (A_{\wedge,\Dom_\wedge}-\lambda)^{-1} K_\wedge(\lambda) \\
&= q_\wedge (A_{\wedge,\Dom_\wedge}-\lambda)^{-1} (A_\wedge-\lambda) \omega(x|\lambda|^{1/m}) \bkappa_{|\lambda/\lambda_0|^{1/m}},
\end{align*}
where $q_\wedge:\Dom_{\wedge,\max}\to \Dom_{\wedge,\max}/\Dom_{\wedge,\min}$ is the quotient map.

For $\lambda\in\bgres(A_\wedge)$ let $\K_{\wedge,\lambda} = \ker(A_{\wedge,\max}-\lambda)$. Then (\cite[Lemma 5.7]{GKM1})
\begin{equation}\label{ResCapBgRes}
\lambda\in \resolv(A_{\wedge,\Dom_\wedge})\; \text{ if and only if }\;\Dom_{\wedge,\max} = \Dom_\wedge \oplus \K_{\wedge,\lambda}
\end{equation}
in which case we let $\pi_{\Dom_\wedge,\K_{\wedge,\lambda}}$ be the projection on $\Dom_\wedge$ according to this decomposition. If $B_{\wedge,\max}(\lambda)$ is the right inverse of $A_{\wedge,\max}-\lambda$ with range $\K_{\wedge,\lambda}^\perp$, then
\begin{equation*}
 (A_{\wedge,\Dom_\wedge}-\lambda)^{-1} = \pi_{\Dom_\wedge,\K_{\wedge,\lambda}}B_{\wedge,\max}(\lambda)
\end{equation*}
and $B_{\wedge,\max}(\lambda)(A_{\wedge,\max}-\lambda)$ is the orthogonal projection onto $\K_{\wedge,\lambda}^\perp$. Thus
\begin{equation*}
 \pi_{\Dom_\wedge,\K_{\wedge,\lambda}}B_{\wedge,\max}(\lambda)(A_{\wedge,\max}-\lambda)= \pi_{\Dom_\wedge,\K_{\wedge,\lambda}}
\end{equation*}
and therefore,
\begin{equation*}
 F_{\wedge,\Dom_\wedge}(\lambda)^{-1} = q_\wedge \,\pi_{\Dom_\wedge,\K_{\wedge,\lambda}}\, \omega(x|\lambda|^{1/m}) \bkappa_{|\lambda/\lambda_0|^{1/m}}.
\end{equation*}
Let 
\begin{equation}\label{DandKlambda}
 D = \Dom_{\wedge}/\Dom_{\wedge,\min}, \quad
 K_{\wedge,\lambda} = \big(\K_{\wedge,\lambda}+\Dom_{\wedge,\min}\big)/\Dom_{\wedge,\min}.
\end{equation}
Again by Lemma 5.7 of \cite{GKM1}, either of the conditions in \eqref{ResCapBgRes} is equivalent to $D\cap K_{\wedge,\lambda}=0$, hence to
\begin{equation}\label{QuotientDecomposition}
 \Dom_{\wedge,\max}/\Dom_{\wedge,\min} = D \oplus K_{\wedge,\lambda}
\end{equation}
by dimensional considerations, since $\dim K_{\wedge,\lambda} = \dim \K_{\wedge,\lambda}=d'$. Let then $\pi_{D,K_{\wedge,\lambda}}$ be the projection on $D$ according to the decomposition \eqref{QuotientDecomposition}. Then $q_\wedge\,\pi_{\Dom_\wedge,\K_{\wedge,\lambda}}=\pi_{D,K_{\wedge,\lambda}}\, q_\wedge$ and
\begin{align} \notag
F_{\wedge,\Dom_\wedge}(\lambda)^{-1}&= \pi_{D,K_{\wedge,\lambda}}q_\wedge \omega(x|\lambda|^{1/m}) \bkappa_{|\lambda/\lambda_0|^{1/m}} \\ \label{FwedgeInvProjection}
&= \pi_{D,K_{\wedge,\lambda}}\bkappa_{|\lambda/\lambda_0|^{1/m}}
\end{align}
since multiplication by $(1-\omega(x|\lambda|^{1/m}))$ maps $\Dom_{\wedge,\max}$ into $\Dom_{\wedge,\min}$ for every $\lambda$.

We will now express $F_{\wedge,D_\wedge}(\lambda)^{-1}$ in terms of projections with $K_{\wedge,\lambda_0}$ in place of $K_{\wedge,\lambda}$. This will of course require replacing $D$ by a family depending on $\lambda$.

Fix $\lambda \in \open \Lambda$, let $S_{\lambda,m}$ be the connected component of $\set{\zeta:\zeta^m\lambda\in \open \Lambda}$ containing $\R_+$. Since $\Lambda\ne \C$, $S_{\lambda,m}$ omits a ray, and so the map $\R_+\ni \varrho \mapsto \bkappa_\varrho \in \Aut(\Dom_{\wedge,\max}/\Dom_{\wedge,\min})$ extends holomorphically to a map
\begin{equation*}
S_{\lambda,m} \ni\zeta \mapsto \bkappa_\zeta\in \Aut(\Dom_{\wedge,\max}/\Dom_{\wedge,\min}).
\end{equation*}
It is an elementary fact that 
\begin{equation*}
\bkappa_{\zeta}^{-1}\big(\pi_{D,K_{\wedge,\lambda}}\big)\bkappa_{\zeta} = \pi_{\bkappa_{\zeta}^{-1} D,\bkappa_{\zeta}^{-1} K_{\wedge,\lambda}}.
\end{equation*}
A simple consequence of \eqref{HomogeneityOfAwedge} is that $\kappa_{\zeta}^{-1} \K_{\wedge,\lambda} = \K_{\wedge,\lambda/\zeta^m}$ if $\zeta\in \R_+$, hence also $\bkappa_{\zeta}^{-1} K_{\wedge,\lambda} = K_{\wedge,\lambda/\zeta^m}$ for such $\zeta$ since the maps $q_\wedge|_{\K_{\wedge,\lambda}}:\K_{\wedge,\lambda}\to K_{\wedge,\lambda}$ are isomorphisms. Therefore
\begin{equation}\label{PiScaling}
\bkappa_{\zeta}^{-1}\big(\pi_{D,K_{\wedge,\lambda}}\big)\bkappa_{\zeta} =  \pi_{\bkappa_{\zeta}^{-1} D,K_{\wedge,\lambda/\zeta^m}}
\end{equation}
if $\zeta\in \R_+$. This formula holds also for arbitrary $\zeta\in S_{\lambda,m}$. To see this we make use of the family of isomorphisms $\mathfrak P(\lambda'):\K_{\wedge,\lambda_0}\to \K_{\wedge,\lambda'}$ (defined for $\lambda'$ in the connected component of $\bgres(A_\wedge)$ containing $\lambda_0$) constructed in Section 7 of \cite{GKM1}. Its two basic properties are that $\lambda'\mapsto \mathfrak P(\lambda')\phi$ is holomorphic for each $\phi\in \K_{\wedge,\lambda_0}$ and that $\kappa_\varrho \mathfrak P(\lambda')=\mathfrak P(\varrho^m \lambda')$ if $\varrho\in \R_+$. These statements are, respectively, Proposition 7.9 and Lemma 7.11 of \cite{GKM1}. Let
\begin{equation*}
f:\Dom_{\wedge,\max} \to \C
\end{equation*}
be an arbitrary continuous linear map that vanishes on $\K_{\wedge,\lambda}$. For any $\phi\in \K_{\wedge,\lambda_0}$ the function
\begin{equation*}
S_{\lambda,m}\ni \zeta\mapsto \langle f, \kappa_\zeta \mathfrak P(\lambda/\zeta^m)\phi\rangle \in \C
\end{equation*}
is holomorphic and vanishes on $\R_+$, the latter because $\kappa_\zeta \mathfrak P(\lambda/\zeta^m)=\mathfrak P(\lambda)$ for such $\zeta$. Therefore $\langle f, \kappa_\zeta \mathfrak P(\lambda/\zeta^m)\phi\rangle=0$ for all $\zeta\in S_{\lambda,m}$. Since $f$ is arbitrary, we must have $\kappa_\zeta \mathfrak P(\lambda/\zeta^m)\phi\in \K_{\wedge,\lambda}$. Therefore $\mathfrak P(\lambda/\zeta^m)\phi\in \kappa_\zeta^{-1} \K_{\wedge,\lambda}$. Since $\mathfrak P(\lambda/\zeta^m):\K_{\wedge,\lambda_0} \to \K_{\wedge,\lambda/\zeta^m}$ is an isomorphism, $\K_{\wedge,\lambda/\zeta^m}=\kappa_\zeta^{-1} \K_{\wedge,\lambda}$ when $\zeta\in S_{\lambda,m}$. This shows
\begin{equation*}
K_{\wedge,\lambda/\zeta^m}=\kappa_\zeta^{-1} K_{\wedge,\lambda}
\end{equation*}
and hence that \eqref{PiScaling} holds for $\zeta\in S_{\lambda,m}$.

The principal branch of the $m$-th root gives a bijection
\begin{equation}\label{TheRoot}
(\,\cdot \,)^{1/m}:\lambda_0^{-1} \open \Lambda\to S_{\lambda_0,m}.
\end{equation}
The reader may now verify that with this root, and with the notation $\hat \zeta=\zeta/|\zeta|$ whenever $\zeta\in \C\minus 0$, one has
\begin{equation}\label{FwedgeInvHolProjection}
\bkappa_{|\lambda|^{1/m}}^{-1}F_{\wedge,\Dom_\wedge}(\lambda)^{-1} = \bkappa_{|\lambda_0|^{1/m}}^{-1} \bkappa_{(\hat\lambda/\hat\lambda_0)^{1/m}}
\Big(\pi_{\bkappa_{(\lambda/\lambda_0)^{1/m}}^{-1} D,K_{\wedge,\lambda_0}} \Big) \bkappa_{(\hat\lambda/\hat\lambda_0)^{1/m}}^{-1}
\end{equation}
when $\lambda\in \open \Lambda\cap \resolv(A_{\wedge,\Dom_\wedge})$. The arguments leading to this formula remain valid if $\Lambda$ is replaced by a slightly bigger closed sector, so the formula just proved holds in $(\Lambda\minus 0)\cap \resolv(A_{\wedge,\Dom_\wedge})$.

The projection $\pi_{\bkappa_{(\lambda/\lambda_0)^{1/m}}^{-1}D, K_{\wedge,\lambda_0}}$ is thus a key component of the resolvent of $A_{\wedge,\Dom_\wedge}$ whose behavior for large $|\lambda|$ will be analyzed in Section \ref{sec:AsymptoticsOfProjection} under a certain fundamental condition which happens to be equivalent to the condition that $\Lambda$ is a sector of minimal growth for $A_{\wedge,\Dom_\wedge}$. We now proceed to discuss this condition.

\medskip
The condition that the sector $\Lambda$ with $\Lambda\minus 0 \subset \bgres(A_\wedge)$ is a sector of minimal growth for $A_{\wedge,\Dom_\wedge}$ was shown in \cite[Theorem~8.3]{GKM1} to be equivalent to the invertibility of $A_{\wedge,\Dom_\wedge}-\lambda$ for $\lambda$ in
\begin{equation*}
\Lambda_R=\set{\lambda\in\Lambda:|\lambda|\ge R}
\end{equation*}
together with the uniform boundedness of $\pi_{\kappa_{|\lambda|^{1/m}}^{-1} D,K_{\hat\lambda}}$ in $\Lambda_R$. Further, it was shown in \cite{GKM3} that along a ray containing $\lambda_0$, this condition is in turn equivalent to requiring that the curve
\begin{equation*}
 \varrho\mapsto \bkappa_\varrho^{-1}D: [R,\infty)\to \Gr_{d''}(\Dom_{\wedge,\max}/\Dom_{\wedge,\min})
\end{equation*}
does not approach the set
\begin{equation}\label{Variety}
 \mathscr V_{K_{\wedge,\lambda_0}}=\{D\in \Gr_{d''}(\Dom_{\wedge,\max}/\Dom_{\wedge,\min}): 
 D\cap K_{\wedge,\lambda_0}\ne 0\}
\end{equation}
as $\varrho\to\infty$, a condition conveniently phrased in terms of the limiting set
\begin{multline*}
\Omega^{-}(D)=\big\{D'\in \Gr_{d''}(\Dom_{\wedge,\max}/\Dom_{\wedge,\min}) : \exists\, \varrho_\nu\to\infty \text{ in } \R_+ \\ 
  \text{such that } \bkappa_{\varrho_\nu}^{-1}D\to D' \text{ as } \nu\to\infty\big\}:
\end{multline*}
A ray $\set{r\lambda_0\in\C: r> 0}$ contained in $\bgres(A_\wedge)$ is a ray of minimal growth for $A_{\wedge,\Dom_\wedge}$ if and only if
\begin{equation*}
 \Omega^{-}(D)\cap \mathscr V_{K_{\wedge,\lambda_0}}= \emptyset.
\end{equation*}
Define
\begin{multline}\label{OmegaMinus}
\Omega^{-}_{\Lambda}(D)=\big\{D'\in \Gr_{d''}(\Dom_{\wedge,\max}/\Dom_{\wedge,\min})
: \exists\, \{\zeta_\nu\}_{\nu=1}^{\infty}\subset\C \text{ with } \lambda_0\zeta_\nu \in \Lambda \text{ and} \\
 \;\; |\zeta_\nu| \to\infty \st \bkappa_{\zeta_{\nu}^{1/m}}^{-1}D\to D' \text{ as } \nu\to\infty\big\}.
\end{multline}
in which we are using the holomorphic extension of $\varrho\mapsto \bkappa_\varrho$ to $S_{\lambda_0,m}$ and the $m$-th root is the principal branch, as specified in \eqref{TheRoot}. We can now consolidate all these conditions as follows.

\begin{theorem}\label{EquivRayConditionsWedge}
Let $\Lambda$ be a closed sector such that $\Lambda\minus 0\subset \bgres(A_\wedge)$, let $\lambda_0\in \open\Lambda$. The following statements are equivalent:
\begin{enumerate}[$(i)$]
\item $\Lambda$ is a sector of minimal growth for $A_{\wedge,\Dom_\wedge}$;
\item there are constants $C,\, R>0$ such that $\Lambda_R\subset \resolv(A_{\wedge,\Dom_\wedge})$ and
\begin{equation*}
\Big\|\pi_{\bkappa_{\zeta^{1/m}}^{-1}D,K_{\wedge,\lambda_0}}\Big\|_{\L(\Dom_{\wedge,\max}/ \Dom_{\wedge,\min})}\leq C
\end{equation*} 
for every $\zeta$ such that $\lambda_0\zeta \in\Lambda_R$;
\item $\Omega^{-}_{\Lambda}(D) \cap \mathscr V_{K_{\wedge,\lambda_0}} = \emptyset$.
\end{enumerate}
\end{theorem}
\begin{proof}
By means of \eqref{PiScaling} we get the identity
\begin{equation*}
\pi_{\bkappa_{\zeta^{1/m}}^{-1}D,K_{\wedge,\lambda_0}} =
\bkappa_{\hat\zeta^{1/m}}^{-1} \bkappa_{|\lambda_0|^{1/m}} 
\Big(\pi_{\bkappa_{|\lambda|^{1/m}}^{-1} D,K_{\wedge,\hat\lambda}} \Big)
\bkappa_{|\lambda_0|^{1/m}}^{-1} \bkappa_{\hat\zeta^{1/m}},
\end{equation*}
which is valid for large $\lambda\in \Lambda$, $\zeta=\lambda/\lambda_0$, and $\hat\zeta=\zeta/|\zeta|$. Since $\bkappa_{\hat\zeta^{1/m}}$ and $\bkappa_{\hat\zeta^{1/m}}^{-1}$ are uniformly bounded, \cite[Theorem~8.3]{GKM1} gives that $(i)$ and $(ii)$ are equivalent.

We now prove that $(ii)$ and $(iii)$ are equivalent. Let $\Sing_{\wedge,\max}= \Dom_{\wedge,\max}/\Dom_{\wedge,\min}$ and assume $(iii)$ is satisfied. Since $\Omega^{-}_{\Lambda}(D)$ and $\mathscr V_{K_{\wedge,\lambda_0}}$ are closed sets in $\Gr_{d''}(\Sing_{\wedge,\max})$, there is a neighborhood $\mathcal U$ of $\mathscr V_{K_{\wedge,\lambda_0}}$ and a constant $R>0$ such that if $|\lambda_0\zeta|>R$ then $\bkappa_{\zeta^{1/m}}^{-1}D\not\in \mathcal U$. Let $\delta:\Gr_{d''}(\Sing_{\wedge,\max})\times \Gr_{d'}(\Sing_{\wedge,\max})\to\R$ be as in Section~5 of \cite{GKM1}. Since $\mathscr V_{K_{\wedge,\lambda_0}}$ is the zero set of the continuous function $\mathcal V\mapsto \delta(\mathcal V,K_{\wedge,\lambda_0})$, there is a constant $\delta_0>0$ such that  $\delta(\bkappa_{\zeta^{1/m}}^{-1}D,K_{\wedge,\lambda_0})>\delta_0$ for every $\zeta$ such that $\lambda_0\zeta\in\Lambda_R$. Then \cite[Lemma~5.12]{GKM1} gives $(ii)$.

Conversely, let $(ii)$ be satisfied. Suppose $\Omega^{-}_{\Lambda}(D)\cap \mathscr V_{K_{\wedge,\lambda_0}} \not=\emptyset$ and let $D_0$ be an element in the intersection. 
Thus $D_0\cap K_{\wedge,\lambda_0}\not=\{0\}$ and there is a sequence $\{\zeta_\nu\}_{\nu=1}^{\infty}\subset\C$ with $\lambda_0\zeta_\nu \in \Lambda$ such that $|\zeta_\nu| \to\infty$ and $D_\nu=\bkappa_{\zeta_{\nu}^{1/m}}^{-1}D\to D_0$ as $\nu\to\infty$. 
If $\nu$ is such that $|\lambda_0\zeta_\nu|>R$, then $\lambda_0\zeta_\nu\in \resolv(A_{\wedge,\Dom_{\wedge}})$ and $D\cap K_{\wedge,\lambda_0\zeta_\nu}=\{0\}$, so $D_\nu\cap K_{\wedge,\lambda_0}=\{0\}$. Thus for $\nu$ large enough $D_\nu\not\in \mathscr V_{K_{\wedge,\lambda_0}}$. 

Pick $u\in D_0\cap K_{\wedge,\lambda_0}$ with $\|u\|=1$. Let $\pi_{D_\nu}$ be the orthogonal 
projection on $D_\nu$. Since $D_\nu\to D_0$ as $\nu\to\infty$, we have $\pi_{D_\nu}\to \pi_{D_0}$, so $u_\nu=\pi_{D_\nu}u\to \pi_{D_0}u=u$. For $\nu$ large, $D_\nu\not\in \mathscr V_{K_{\wedge,\lambda_0}}$, so $u_\nu-u\not=0$. Now, since $u_\nu\in D_\nu$, $u\in K_{\wedge,\lambda_0}$, and $u_\nu\to u$, 
\begin{equation*}
\pi_{D_\nu,K_{\wedge,\lambda_0}} 
\bigg(\frac{u_\nu-u}{\|u_\nu-u\|}\bigg) =
   \frac{u_\nu}{\|u_\nu-u\|} \to \infty \text{ as } \nu\to\infty.
\end{equation*}
But this contradicts $(ii)$. Hence $\Omega^{-}(D)\cap \mathscr V_{K_{\wedge,\lambda_0}}= \emptyset$.
\end{proof}

\bigskip
If $\Dom_\wedge$ is not $\kappa$-invariant, the asymptotic analysis of $F_{\wedge,\Dom_\wedge}(\lambda)^{-1}$ (through the analysis of the projection $\pi_{D,K_{\wedge,\lambda}}$) leads to rational functions of the form
\begin{equation}\label{RationalFunction}
 r(\lambda^{i\mu_{1}},\dots,\lambda^{i\mu_{N}},\log \lambda) 
 =\frac{p(\lambda^{i\mu_{1}},\dots,\lambda^{i\mu_{N}}, \log \lambda)}{q(\lambda^{i\mu_{1}},\dots,\lambda^{i\mu_{N}},\log \lambda)}
\end{equation}
with $\mu_{\ell}\in\R$ for $\ell=1,\dots,N$, where $q(z_1,\dots,z_{N+1})$ is a polynomial over $\C$ such that $|q(\lambda^{i\mu_{1}},\dots,\lambda^{i\mu_{N}}, \log \lambda)|>\delta$ for some $\delta>0$ and every sufficiently large $\lambda\in \Lambda$, and 
\begin{equation*}
 p(\lambda^{i\mu_{1}},\dots,\lambda^{i\mu_{N}},\log \lambda)=
 \sum_{\alpha,k} a_{\alpha k}(\lambda) \lambda^{i\alpha\mu}\log^k\lambda
\end{equation*}
with $\mu=(\mu_1,\dots,\mu_N)$, $\alpha\in\N_0^N$, $k\in\N_0$, and coefficients
\[ a_{\alpha k}\in C^\infty(\Lambda\minus 0, \L(\Dom_{\wedge}/\Dom_{\wedge,\min},\Dom_{\wedge,\max}/\Dom_{\wedge,\min}))
\]
such that $a_{\alpha k}(\varrho^m\lambda)=\bkappa_\varrho a_{\alpha k}(\lambda)$ for every $\varrho>0$.

\begin{proposition}\label{FwedgeInvSymbol}
If $\Lambda$ is a sector of minimal growth for $A_{\wedge,\Dom_\wedge}$, then for $R>0$ large enough, the family $F_{\wedge,\Dom_\wedge}(\lambda)= F_\wedge(\lambda)|_{\Dom_\wedge/\Dom_{\wedge,\min}}$ is invertible for $\lambda\in\Lambda_R$ and
$F_{\wedge,\Dom_\wedge}(\lambda)^{-1}$ has the following properties:
\begin{enumerate}[$(i)$]
\item $F_{\wedge,\Dom_\wedge}(\lambda)^{-1} \in C^\infty(\Lambda_R; \L(\Dom_{\wedge}/\Dom_{\wedge,\min},\Dom_{\wedge,\max}/\Dom_{\wedge,\min}))$, and for every $\alpha,\beta\in\N_0$ we have
\begin{equation}\label{FwedgeInvSymbolEstimate}
\norm{\bkappa_{|\lambda|^{1/m}}^{-1}\partial_\lambda^\alpha \partial_{\bar\lambda}^\beta\, F_{\wedge,\Dom_\wedge}(\lambda)^{-1}} = \Oh(|\lambda|^{\frac{\nu}{m}-\alpha-\beta}) \;\text{ as }\; |\lambda|\to\infty,
\end{equation}
with $\nu=0$;
\item for all $j\in\N_0$ there exist rational functions $r_j$ of the form \eqref{RationalFunction} and a decreasing sequence of real numbers $0=\nu_0> \nu_1> \cdots \to -\infty$ such that for every $J\in \N$, the difference
\begin{equation}\label{FwedgeInvAsympExp}
 F_{\wedge,\Dom_\wedge}(\lambda)^{-1}- \sum_{j=0}^{J-1} r_j(\lambda^{i\mu_{1}},\dots,\lambda^{i\mu_{N}},\log\lambda)\, \lambda^{\nu_j/m}
\end{equation}
satisfies \eqref{FwedgeInvSymbolEstimate} with $\nu=\nu_J+\eps$ for any $\eps>0$.
\end{enumerate}
\end{proposition}

The phases $\mu_1,\dots,\mu_N$, and the exponents $\nu_j$ in \eqref{FwedgeInvAsympExp} depend on the boundary spectrum of $A$. In fact, $\mu_1,\dots,\mu_N\in\M$ and $\nu_j\in\E$ for all $j$,
see \eqref{PhasesMainThm} and \eqref{SemigroupDef}.

This suggests the introduction of operator valued symbols with a notion of asymptotic expansion in components that take into account the above rational structure and the $\kappa$-homogeneity of their numerators. The idea of course is to have a class of symbols whose structure is preserved under composition, differentiation, and asymptotic summation. In the appendix we propose such a class, $S_{\r}^{\nu^+}(\Lambda; E,\tilde E)$, a subclass of the operator-valued symbols $S^{\infty}(\Lambda; E,\tilde E)$ introduced by Schulze, where $E$ and $\tilde E$ are Hilbert spaces equipped with suitable group actions. 
The space $S_{\r}^{\nu^+}(\Lambda; E,\tilde E)$ is contained in $S^{\nu+\eps}(\Lambda; E,\tilde E)$ for any $\eps>0$.

As reviewed at the beginning of the appendix, the notion of anisotropic homogeneity in $S^{(\nu)}(\Lambda; E,\tilde E)$ depends on the group actions in $E$ and $\tilde E$. Thus homogeneity is always to be understood with respect to these actions. 

In the symbol terminology, we have
\begin{equation*}
F_{\wedge,\Dom_\wedge}(\lambda)^{-1} \in \big(S_\r^{0^+}\cap S^0\big) (\Lambda_R;\Dom_{\wedge}/\Dom_{\wedge,\min},\Dom_{\wedge,\max}/\Dom_{\wedge,\min}),
\end{equation*}
where $\Dom_{\wedge}/\Dom_{\wedge,\min}$ carries the trivial action and $\Dom_{\wedge,\max}/\Dom_{\wedge,\min}$ is equipped with $\bkappa_\varrho$.

\begin{proof}[Proof of Proposition~\ref{FwedgeInvSymbol}]
Since $\Lambda$ is a sector of minimal growth for $A_{\wedge,\Dom_\wedge}$, there exists $R>0$ such that $(A_{\wedge,\Dom_\wedge}-\lambda)$ is invertible for $\lambda\in\Lambda_R$, which by definition is equivalent to the invertibility of $F_{\wedge,\Dom_\wedge}(\lambda)$. Since the map $\zeta\mapsto \bkappa_{\hat\zeta^{1/m}}$ is uniformly bounded (recall that $\hat\zeta= \zeta/|\zeta|$), the relation \eqref{FwedgeInvHolProjection} together with Theorem~\ref{EquivRayConditionsWedge} give the estimate \eqref{FwedgeInvSymbolEstimate} for $\alpha=\beta=0$. If we differentiate with respect to $\lambda$ (or $\bar\lambda$), then
\begin{align*}
\partial_\lambda F_{\wedge,\Dom_\wedge}(\lambda)^{-1} &= - F_{\wedge,\Dom_\wedge}(\lambda)^{-1} \big[\partial_\lambda F_{\wedge,\Dom_\wedge}(\lambda)\big] F_{\wedge,\Dom_\wedge}(\lambda)^{-1} \\
&= - F_{\wedge,\Dom_\wedge}(\lambda)^{-1} \big[\partial_\lambda F_{\wedge}(\lambda)\big] F_{\wedge,\Dom_\wedge}(\lambda)^{-1}.
\end{align*}
Now, if we equip $\Dom_{\wedge}/\Dom_{\wedge,\min}$ with the trivial group action and $\Dom_{\wedge,\max}/\Dom_{\wedge,\min}$ with $\bkappa_\varrho$, then $F_\wedge(\lambda):\Dom_{\wedge,\max}/\Dom_{\wedge,\min} \to \Dom_{\wedge}/\Dom_{\wedge,\min}$ is homogeneous of degree zero, hence $\norm{\partial_\lambda F_{\wedge}(\lambda) \bkappa_{|\lambda|^{1/m}}}$ is $\Oh(|\lambda|^{-1})$ as $|\lambda|\to\infty$. Therefore, 
\[ \norm{\bkappa_{|\lambda|^{1/m}}^{-1}\partial_\lambda F_{\wedge,\Dom_\wedge}(\lambda)^{-1}}
   = \Oh(|\lambda|^{-1}) \;\text{ as } |\lambda|\to\infty, \]
since $\bkappa_{|\lambda|^{1/m}}^{-1}\partial_\lambda F_{\wedge,\Dom_\wedge}(\lambda)^{-1}$ can be written as
\begin{equation*}
 -\big[\bkappa_{|\lambda|^{1/m}}^{-1} F_{\wedge,\Dom_\wedge}(\lambda)^{-1}\big] \big[\partial_\lambda F_{\wedge}(\lambda) \bkappa_{|\lambda|^{1/m}}\big] \big[\bkappa_{|\lambda|^{1/m}}^{-1} F_{\wedge,\Dom_\wedge}(\lambda)^{-1}\big],
\end{equation*}
and the first and last factors are uniformly bounded by our previous argument. The corresponding estimates for arbitrary derivatives follow by induction.

Next, observe that by \eqref{FwedgeInvHolProjection},
\begin{equation*}
 F_{\wedge,\Dom_\wedge}(\lambda)^{-1} = \bkappa_{\zeta^{1/m}}
\Big(\pi_{\bkappa_{\zeta^{1/m}}^{-1} D,K_{\wedge,\lambda_0}} \Big) \bkappa_{\hat\zeta^{1/m}}^{-1}
\end{equation*}
with $\zeta=\lambda/\lambda_0$ and $\hat\zeta= \zeta/|\zeta|$. For $\lambda\in\Lambda_R$ let $k(\lambda)=\bkappa_{\zeta^{1/m}}$ and $\hat k(\lambda)=\bkappa_{\hat\zeta^{1/m}}^{-1}$. Then $k(\lambda)$ is a homogeneous symbol in $S^{(0)}(\Lambda_R; \Dom_{\wedge,\max}/\Dom_{\wedge,\min},\Dom_{\wedge,\max}/\Dom_{\wedge,\min})$, where the first copy of the quotient is equipped with the trivial action and the target space carries $\bkappa_\varrho$. Similarly, $\hat k(\lambda)\in S^{(0)}(\Lambda_R; \Dom_{\wedge}/\Dom_{\wedge,\min},\Dom_{\wedge,\max}/\Dom_{\wedge,\min})$ with respect to the trivial action on both spaces. 

Finally, the asymptotic expansion claimed in $(ii)$ follows from Theorem~\ref{AsymptOfProjB} together with the homogeneity properties of $k(\lambda)$ and $\hat k(\lambda)$.
\end{proof}

As a consequence of Proposition~\ref{FwedgeInvSymbol}, and since $B_\wedge(\lambda)$, $[1-B_\wedge(\lambda)(A_\wedge-\lambda)]$, and $T_\wedge(\lambda)$ in \eqref{ResolventWedgeStructure} are homogeneous of degree $-m$, $0$, and $-m$, in their respective classes, we obtain:

\begin{proposition}\label{ResolventAwedge}
If $\Lambda$ is a sector of minimal growth for $A_{\wedge,\Dom_\wedge}$, then for $R>0$ large enough, we have
\begin{equation*}
(A_{\wedge,\Dom_\wedge}-\lambda)^{-1} \in 
\big(S_{\r}^{(-m)^+}\cap S^{-m}\big)(\Lambda_R;x^{-m/2}L^2_b,\Dom_{\wedge,\max}),  
\end{equation*}
where the spaces are equipped with the standard action $\kappa_\varrho$. The components have orders $\nu^+$ with $\nu \in \E$ and their phases belong to $\M$, see \eqref{PhasesMainThm} and \eqref{SemigroupDef}.
\end{proposition}

\section{Limiting Orbits}
\label{sec:LimitingOrbits}

We will write $\Sing$ instead of $\Dom_{\wedge,\max}/\Dom_{\wedge,\min}$ and denote by $\a:\Sing\to\Sing$ the infinitesimal generator of the $\R_+$ action $(\varrho,v) \mapsto \bkappa_{\varrho}^{-1}v$ on $\Sing$, so that $\bkappa_{\varrho}^{-1}D = e^{t\a} D$ with $t=\log\varrho$. In what follows we allow $t$ to be complex. The spectrum of $\a$ is related to the boundary spectrum of $A$ by
\begin{equation}\label{TheSpectrum}
\spec\a = \set {-\im \sigma-m/2:\sigma\in \spec_b(A),\ -m/2<\Im\sigma<m/2}.
\end{equation}

For each $\lambda\in \spec\a$ let $\Sing_\lambda$ be the generalized eigenspace of $\a$ associated with $\lambda$, let $\pi_\lambda:\Sing\to\Sing$ be the projection on $\Sing_\lambda$ according to the decomposition 
\begin{equation*}
\Sing=\bigoplus_{\lambda\in\spec \a} \Sing_\lambda.
\end{equation*}
Define $N:\Sing\to \Sing$ and $N_\lambda:\Sing_\lambda \to \Sing_\lambda$ by 
\begin{equation*}
N=\a-\sum_{\lambda\in\spec \a} \lambda\pi_\lambda,\qquad N_\lambda= N|_{\Sing_\lambda},
\end{equation*}
respectively, and let 
\begin{equation}\label{aPrime}
\a':\Sing\to \Sing, \quad \a'=\sum_{\lambda\in\spec \a} (\im \Im \lambda)\pi_\lambda.
\end{equation}
For $\mu\in \Re(\spec \a)$ let 
\begin{equation*}
\tilde \Sing_\mu=\bigoplus_{\substack{\lambda\in \spec a\\\Re \lambda=\mu}} \Sing_{\lambda},
\end{equation*}
let $\tilde \pi_\mu:\Sing\to\Sing$ be the projection on $\tilde \Sing_\mu$ according to the decomposition
\begin{equation*}
\Sing=\bigoplus_{\mu\in \Re(\spec \a)} \tilde \Sing_{\mu},
\end{equation*}
and let 
\begin{equation*}
\tilde N_\mu=N|_{\tilde \Sing_\mu}:\tilde \Sing_\mu\to\tilde \Sing_\mu.
\end{equation*}
Fix an auxiliary Hermitian inner product on $\Sing$ so that $\bigoplus \Sing_{\lambda}$ is an orthogonal decomposition of $\Sing$. Then $\a'$ is skew-adjoint and $e^{t\a'}$ is unitary if $t$ is real.

\begin{proposition}\label{ExistenceOfShadow}
For every $D \in \Gr_{d''}(\Sing)$ there is $D_\infty\in \Gr_{d''}(\Sing)$ such that
\begin{equation}\label{LimitingSpace}
\dist(e^{t\a}D, e^{t\a'}D_\infty)\to 0 \text{ as }\Re t\to\infty\text{ in }\Strip_\theta=\set{t\in \C:|\Im t|\leq \theta}
\end{equation}
for any $\theta>0$. The set
\begin{equation*}
\Omega^+_{\theta}(D)=\set{D'\in \Gr_{d''}(\Sing):\exists\set{t_\nu}\subset \Strip_\theta:\Re t_\nu\to \infty \text{ and }\lim_{\nu\to\infty}e^{t_\nu \a}D=D'}
\end{equation*}
is the closure of 
\begin{equation*}
\set{e^{t\a'}D_\infty:t\in \Strip_\theta}.
\end{equation*}
\end{proposition}

We are using $\Omega^+$ for the limit set for consistency with common usage: we are letting $\Re t$ tend to infinity.

If $\F$ is a vector space, we will write $\F[t,t^{-1}]$ for the space of polynomials in $t$ and $t^{-1}$ with coefficients in $\F$ (\ie, the $\F$-valued rational functions on $\C$ with pole only at $0$). If $p\in \F[t,t^{-1}]$, let $\coeff_s(p)$ denote the coefficient of $t^s$ in $p$, and if $p\ne 0$, let 
\begin{equation*}
\ord(p)=\max\set {s\in \mathbb Z: \coeff_s(p)\ne 0}.
\end{equation*}

The proof of the proposition hinges on the following lemma.

\begin{lemma}\label{ExistenceOfShadowLemma}
Let $D\subset \Sing$ be an arbitrary nonzero subspace. Define $D^1=D$ and by induction define
\begin{equation*}
\mu_\ell=\max\set{\mu\in \Re(\spec \a):\tilde \pi_\mu D^\ell\ne 0}, \quad D^{\ell+1}=\ker \tilde \pi_{\mu_\ell}|_{D^\ell},\quad D_{\mu_\ell}=(D^{\ell+1})^\perp\cap D^\ell
\end{equation*}
starting with $\ell=1$. Let $L$ be the smallest $\ell$ such that $D^{\ell+1}=0$. Thus
\begin{equation}\label{PiEllIso}
\tilde \pi_{\mu_\ell}|_{D^\ell}:D_{\mu_\ell}\to \tilde \pi_{\mu_\ell}D_{\mu_\ell}\text{ is an isomorphism }
\end{equation}
and $D=\bigoplus_{\ell=1}^L D_{\mu_\ell}$. Then for each $\ell$ there are  elements
\begin{equation*}
\tilde p^\ell_k\in \tilde \pi_{\mu_\ell}D_{\mu_\ell}[t,1/t], \quad k=1,\dotsc,\dim D_{\mu_\ell},
\end{equation*}
such that with
\begin{equation*}
\tilde q^\ell_k(t) = e^{t \tilde N_{\mu_\ell}} \tilde p^\ell_k(t)
\end{equation*}
we have that $\ord \tilde q^\ell_k=0$ and the elements
\begin{equation*}
g^\ell_k=\coeff_0(\tilde q^\ell_k), \quad k=1,\dotsc,\dim D_{\mu_\ell}, \text{ are  independent}.
\end{equation*}
\end{lemma}

The proof will be given later.

\begin{proof}[Proof of Proposition \ref{ExistenceOfShadow}]
Suppose $D\subset \Sing$ is a subspace. With the notation of Lemma \ref{ExistenceOfShadowLemma} let 
\begin{equation*}
D_{\mu_\ell,\infty}=\LinSpan\set{g^\ell_k: k=1,\dotsc,\dim D_{\mu_\ell}}.
\end{equation*}
Since $e^{t\tilde N_{\mu_\ell}}$ is invertible and $\tilde q^\ell_k(t)=g^\ell_k+\tilde h^\ell_k(t)$ with $\tilde h^\ell_k(t)=\Oh(t^{-1})$ for large $\Re t$ ($t\in \Strip_\theta$), the vectors $\tilde p^\ell_k(t)$ form a basis of $\tilde \pi_{\mu_\ell}D_{\mu_\ell}$ for all sufficiently large $t$. Using \eqref{PiEllIso} we get unique elements 
\begin{equation*}
p^\ell_k\in D_{\mu_\ell}[t,1/t],\quad \tilde \pi_{\mu_\ell}p^\ell_k=\tilde p^\ell_k.
\end{equation*}
For each $\ell$ the $p^\ell_k(t)$ give a basis of $D_{\mu_\ell}$ if $t$ is large enough, and therefore also the
\begin{equation*}
e^{-t\mu_\ell}p^\ell_k(t), \quad k=1,\dotsc,\dim D_{\mu_\ell},
\end{equation*}
form a basis of $D_{\mu_\ell}$ for large $\Re t$. Consequently, the vectors
\begin{equation*}
e^{t\a}e^{-t\mu_\ell}p^\ell_k(t), \quad k=1,\dotsc,\dim D_{\mu_\ell},\ \ell=1,\dotsc, L,
\end{equation*}
form a basis of $e^{t\a}D$ for large $\Re t$. We have, with $N_\lambda=N|_{\Sing_\lambda}$,
\begin{align*}
e^{t\a}e^{-t\mu_\ell}p^\ell_k(t) &=\sum_{\lambda\in \spec \a} e^{t(\lambda -\mu_\ell)}e^{t N_\lambda}\pi_\lambda p^\ell_k(t)\\
&=\sum_{\substack{\lambda\in \spec \a\\\Re\lambda=\mu_\ell}} e^{t(\lambda -\mu_\ell)}e^{tN_\lambda}\pi_\lambda p^\ell_k(t) +\sum_{\substack{\lambda\in \spec \a\\\Re\lambda<\mu_\ell}} e^{t(\lambda -\mu_\ell)}e^{tN_\lambda}\pi_\lambda p^\ell_k(t)\\
&=e^{t\a'} e^{t\tilde N_{\mu_\ell}} \tilde \pi_{\mu_\ell} p^\ell_k(t) +\sum_{\substack{\lambda\in \spec \a\\\Re\lambda<\mu_\ell}} e^{t(\lambda -\mu_\ell)}e^{tN_\lambda}\pi_\lambda p^\ell_k(t)\\
&=e^{t\a'} (g^\ell_k+\tilde h^\ell_k(t)) +\sum_{\substack{\lambda\in \spec \a\\\Re\lambda<\mu_\ell}} e^{t(\lambda -\mu_\ell)}e^{tN_\lambda}\pi_\lambda p^\ell_k(t)
\end{align*}
so $e^{t\a}e^{-t\mu_\ell}p^\ell_k(t)=e^{t\a'}g^\ell_k+h^\ell_k(t)$ where $h^\ell_k(t)=\Oh(t^{-1})$ as $\Re t\to\infty$ in $\Strip_\theta$. It follows that \eqref{LimitingSpace} holds with $D_\infty=\bigoplus_{\ell=1}^L D_{\mu_\ell,\infty}$. This completes the proof of the first assertion of Proposition \ref{ExistenceOfShadow}.

\begin{remark}\label{RemarkAlignedSpec}
The formulas for the $v_{k}^\ell(t)=e^{t\a}e^{-t\mu_\ell}p^\ell_k(t)$ given in the last displayed line above will eventually give the asymptotics of the projections $\pi_{e^{t\a}D,K}$ (assuming $\mathscr V_K\cap \Omega^+(D)=\emptyset$, see Theorem~\ref{EquivRayConditionsWedge}). Note that the shift by $m/2$ in \eqref{TheSpectrum} is irrelevant and that the coefficients of the exponents in the formula for $v_k^\ell(t)$ belong to
\begin{equation}\label{GeneratorOfSemiGroup-a}
\set{\lambda-\Re\lambda':\lambda,\lambda'\in \spec\a,\ \Re\lambda\leq \Re\lambda'}
\end{equation}
Because of \eqref{TheSpectrum}, this set is equal to
\begin{equation}\label{GeneratorOfSemiGroup}
-\im \set{\sigma -\im\Im\sigma':\sigma,\sigma'\in \spec_b(A),\ -m/2<\Im\sigma \leq \Im\sigma'<m/2}.
\end{equation}
If all elements of $\set{\sigma\in \spec_b(A):-m/2<\Im\sigma<m/2}$ have the same real part, then all elements of \eqref{GeneratorOfSemiGroup-a} have the same imaginary part $\nu$, the operator $\a'$ is multiplication by $\im \nu$, and we can divide each of the $v_k^\ell(t)$ by $e^{\im t \nu}$ to obtain a basis of $e^{t\a}D$ in which the coefficients of the exponents are all real.
\end{remark}

To prove the second assertion of the proposition, we note first that \eqref{LimitingSpace} implies that $\Omega^+_\theta(D)$ is contained in the closure of $\smallset{e^{t\a'}D_\infty :t\in \Strip_\theta}$. To prove the opposite inclusion, it is enough to show that
\begin{equation}\label{OppInc}
e^{t\a'}D_\infty\in \Omega^+_\theta(D)
\end{equation}
for each $t\in \Strip_\theta$, since $\Omega^+_\theta(D)$ is a closed set. Writing $e^{t\a'}D_\infty$ as $e^{\im \Im t\,\a'}(e^{\Re t\,\a'}D_\infty)$ further reduces the problem to the case  $\theta=0$ (that is, $t$ real). While proving \eqref{OppInc} we will also show that the closure  $\mathcal X$ of $\smallset{e^{t \a'}D_\infty:t\in \R}$ is an embedded torus, equal to $\Omega^+_0(D)$. 

Let $\set{\lambda_k}_{k=1}^K$ be an enumeration of the elements of $\spec \a$. Define $f:\R^K\times \Gr_{d''}(\Sing)\to \Gr_{d''}(\Sing)$ by
\begin{equation*}
f(\tau,D)=e^{\sum \im \tau^k\pi_{\lambda_k}}D,
\end{equation*}
$\tau=(\tau^1,\dotsc,\tau^K)$. This is a smooth map. Since the $\pi_{\lambda_k}$ commute with each other, $f$ defines a left action of $\R^K$ on $\Gr_{d''}(\Sing)$. For each $\tau\in \R^K$ define
\begin{equation*}
f_\tau:\Gr_{d''}(\Sing)\to \Gr_{d''}(\Sing),\quad f_\tau(D)=f(\tau,D)
\end{equation*}
and for each $D\in \Gr_{d''}(\Sing)$ let 
\begin{equation*}
f^D:\R^K\to \Gr_{d''}(\Sing),\quad f^D(\tau)=f(\tau,D).
\end{equation*}
The maps $f_\tau$ are diffeomorphisms.

We claim that $f^{D_\infty}$ factors as the composition of a smooth group homomorphism $\phi:\R^K\to\mathbb T^{K'}$ onto a torus and an embedding $h:\mathbb T^{K'}\to \Gr_{d''}(\Sing)$,
\begin{center}
\begin{picture}(55,55)
\put(-15,43){$\R^K$}
\put(46,43){$\mathbb T^{K'}$}
\put(-25,3){$\Gr_{d''}(\Sing)$.}
\put(-12,38){\vector(0,-1){23}}
\put(-32,26){\footnotesize $f^{D_\infty}$}
\put(16,53){\footnotesize $\phi$}
\put(3,46){\vector(1,0){37}}
\put(25,18){\footnotesize $h$}
\put(40,38){\curve(-28,-21,0,0)}
\put(9,15){\vector(-4,-3){1}}
\end{picture}
\end{center}
Both $\phi$ and $h$ depend on $D_\infty$.

To prove the claim we begin by observing that $\smallset{u\in T\R^K: df^{D_\infty}(u)=0}$ is translation-invariant. Indeed, let $\tau_0\in \R^K$, let $v=(v^1,\dotsc,v^K)\in \R^K$, and let $\gamma:\R\to\R^K$ be the curve $\gamma(t)=tv$. Then 
\begin{equation*}
f^{D_\infty}(\tau_0+\gamma(t))=f_{\tau_0}\circ f^{D_\infty}(\gamma(t))
\end{equation*}
so 
\begin{equation*}
df^{D_\infty}(\sum v^k\partial_{\tau^k}|_{\tau_0}) = df_{\tau_0}\circ df^{D_\infty}(\sum v^k\partial_{\tau^k}|_0).
\end{equation*}
Since $f_{\tau_0}$ is a diffeomorphism,
\begin{multline*}
\sum v^k\partial_{\tau^k}|_{\tau_0}\in [\ker df^{D_\infty}:T_{\tau_0}\R^K\to T_{f^{D_\infty}(\tau_0)} \Gr_{d''}(\Sing)] \\
\iff \sum v^k\partial_{\tau^k}|_0\in [\ker df^{D_\infty}:T_0\R^K\to T_{D_\infty} \Gr_{d''}(\Sing)].
\end{multline*}
Thus the kernel of $df^{D_\infty}$ is translation-invariant as asserted.

Identify the kernel of $d f^{D_\infty}:T_0\R^K\to T_{D_\infty}\Gr_{d''}(\Sing)$ with a subspace $\mathcal S$ of $\R^K$ in the standard fashion. Then $f^{D_\infty}$ is constant on the translates of $\mathcal S$ and if $\mathcal R$ is a subspace of $\R^K$ complementary to $\mathcal S$, then $f^{D_\infty}|_{\mathcal R}$ is an immersion. Renumbering the elements of $\spec \a'$ (and reordering the components of $\R^K$ accordingly) we may take $\mathcal R=\R^{K'}\times 0$.

Since $f^{D_\infty}|_{\mathcal R}$ is an immersion, the sets $\mathcal F_{D'}=\set{\tau\in \mathcal R: f^{D_\infty}(\tau)=D'}$ are discrete for each $D'\in f^{D_\infty}(\mathcal R)$. Using again the property $f^{D_\infty}(\tau_1+\tau_2)=f_{\tau_1}\circ f^{D_\infty}(\tau_2)$ for arbitrary $\tau_1$, $\tau_2\in \R^K$, we see that $\mathcal F_{D_\infty}$ is an additive subgroup of $\mathcal R$ and that $f^{D_\infty}$ is constant on the lateral classes of $\mathcal F_{D_\infty}$. Therefore $f^{D_\infty}|_{\mathcal R}$ factors through a (smooth) homomorphism $\phi:\mathcal R\to \mathcal R/\mathcal F_{D_\infty}$ and a continuous map $\mathcal R/\mathcal F_{D_\infty}\to \Gr_{d''}(\mathcal E)$. Since $f^{D_\infty}$ is $2\pi$-periodic in all variables, $2\pi \mathbb Z^{K'}\subset \mathcal F_{D_\infty}$, so $\mathcal R/\mathcal F_{D_\infty}$ is indeed a torus $\mathbb T^{K'}$. Since $\phi$ is a local diffeomorphism and $f^{D_\infty}$ is smooth, $h$ is smooth.

With this, the proof of the second assertion of the proposition goes as follows. Let $L\subset \R^K$ be the subspace generated by $(\Im\lambda_1,\dotsc,\Im\lambda_K)$. This is a line or the origin. Its image by $\phi$ is a subgroup $H$ of $\mathbb T^{K'}$, so the closure of $\phi(L)$ is a torus $G\subset \mathbb T^{K'}$, and $h(\overline{\phi(L)})$ is an embedded torus $\mathcal X\subset \Gr_{d''}(\Sing)$. On the other hand, $h\circ \phi(L)=f^{D_\infty}(L)$ is the image of the curve $\gamma:t\to e^{t\a'}D_\infty$, so the closure of the image of $\gamma$ is $\mathcal X$. Clearly, $\Omega^+_0(D)\subset \mathcal X$. The equality of $\Omega^+_0(D)$ and $ \mathcal X$ is clear if $\gamma$ is periodic or $L=\set{0}$. So assume that $\gamma$ is not periodic and $L\ne\set{0}$. Then $H\ne G$ and there is a sequence $\set{g_\nu}_{\nu=1}^\infty\subset G\minus H$ such that $g_\nu\to e$, the identity element of $G$. Let $v$ be an element of the Lie algebra of $G$ such that $H$ is the image of $t\mapsto \exp(tv)$. For each $\nu$ there is a sequence $\set{t_{\nu,\rho}}_{\rho=1}^\infty$, necessarily unbounded because $g_\nu\notin H$, such that $g_\nu = \lim_{\rho\to \infty}\exp(t_{\nu,\rho}v)$. We may assume that $\set{t_{\nu,\rho}}_{\rho=1}^\infty$ is monotonic, so it diverges to $+\infty$ or to $-\infty$. In the latter case we replace $g_\nu$ by its group inverse, so we may assume that $\lim_{\rho\to \infty} t_{\nu,\rho}=\infty$ for all $\nu$. Thus if $g\in H$ is arbitrary, then $h(g g_\nu)\in \Omega^+_0(D)$ and $h(g g_\nu)$ converges to $h(g)$. Since $\Omega^+_0(D)$ is closed, this shows that $h\circ \phi(H)\subset \Omega^+_0(D)$. Consequently, also $\mathcal X\subset \Omega^+_0(D)$.

This completes the proof of the second assertion of Proposition~\ref{ExistenceOfShadow}.
\end{proof}

As a consequence of the proof we have that $\Omega^+_\theta(D)$ is a union of embedded tori:
\begin{equation*}
\Omega^+_\theta(D)=\bigcup_{s\in [-\theta,\theta]} e^{\im s\a'}\overline{\set{e^{t \a'}D_\infty:t\in \R}}.
\end{equation*}

The proof of Lemma \ref{ExistenceOfShadowLemma} will be based on the following lemma. The properties of the elements $\tilde p^\ell_k\in \tilde \pi_{\mu_\ell} D_{\mu_\ell}[t,1/t]$ whose existence is asserted in Lemma \ref{ExistenceOfShadowLemma} pertain only $\tilde \Sing_{\mu_\ell}$, $\tilde N_{\mu_\ell}$, and the subspace $\tilde \pi_{\mu_\ell}D_{\mu_\ell}$ of $\tilde \Sing_{\mu_\ell}$. For the sake of notational simplicity we let $\W=\tilde \pi_{\mu_\ell}D_{\mu_\ell}$ and drop the $\mu_\ell$ from the notation. The space $\tilde \Sing$ comes equipped with some Hermitian inner product, and $\tilde N$ is nilpotent.

\begin{lemma}\label{BasicLemma}
There is an orthogonal decomposition
\begin{equation*}
\W=\bigoplus_{j=0}^J\bigoplus_{m=0}^{M_j} \W^m_j
\end{equation*}
$($with nontrivial summands$)$ and nonzero elements
\begin{equation}\label{BasicLemma0}
P^m_j\in \Hom(\W_{j,m},\W_j')[t,t^{-1}]
\end{equation}
where 
\begin{equation*}
\W_{j,m}=\bigoplus_{m'=m}^{M_j} \W^{m'}_j,\quad \W_j'=\bigoplus_{j'=0}^j\W_{j',0},
\end{equation*}
satisfying the following properties.
\begin{enumerate}
\item \label{BasicLemmaItem1} $P^0_j=\id_{\W_{j,0}}$.
\item \label{BasicLemmaItem2} Let
\begin{equation*}
Q^m_j(t)=e^{t\tilde N}P^m_j(t),\quad n^m_j=\ord(Q^m_j).
\end{equation*}
Then the sequence $\set{n^m_j}_{m=0}^{M_j}$ is strictly decreasing and consists of nonnegative numbers.
\item \label{BasicLemmaItem3} Let
\begin{equation}\label{BasicLemma2}
G^m_j=\coeff_{n^m_j}(Q^m_j),\quad \Vee^m_j=G^m_j(\W^m_j).
\end{equation}
Then
\begin{equation}\label{BasicLemma3}
\W_{j,m+1}=(G^m_j)^{-1}\bigg(\bigoplus_{j'=0}^{j-1}\bigoplus_{m'=0}^{M_{j'}}\Vee^{m'}_{j'} + \bigoplus_{m'=0}^{m-1} \Vee^{m'}_j\bigg).
\end{equation}
\item \label{BasicLemmaItem4} There are unique maps $F^{m',m+1}_{j',j}:\W_{j,m+1}\to \W^{m'}_{j'}$ such that 
\begin{equation}\label{BasicLemma4}
G^m_j + \sum_{j'=0}^{j-1}\sum_{m'=0}^{M_{j'}}G^{m'}_{j'}F^{m',m+1}_{j',j} + \sum_{m'=0}^{m-1} G^{m'}_j F^{m',m+1}_{j,j}=0
\end{equation}
holds on $\W_{j,m+1}$, and
\begin{equation}\label{BasicLemma5}
P^{m+1}_j=
P^m_j + \sum_{j'=0}^{j-1}\sum_{m'=0}^{M_{j'}} t^{n^m_j-n^{m'}_{j'}}P^{m'}_{j'} F^{m',m+1}_{j',j} + \sum_{m'=0}^{m-1} t^{n^m_j-n^{m'}_{j}}P^{m'}_j F^{m',m+1}_{j,j}.
\end{equation}
\end{enumerate}
\end{lemma}

The lemma is a definition by induction if we adopt the convention that spaces with negative indices and summations where the upper index is less than the lower index are the zero space. In the inductive process that will constitute the proof of the lemma we will first define $\W_{j,m+1}\subset \W_{j,m}$ using \eqref{BasicLemma3} starting with suitably defined spaces $\W_{j,0}$ and then define $\W^m_j=\W_{j,m}\cap \W_{j,m+1}^\perp$. Note that the right hand side of \eqref{BasicLemma3} depends only on $\W_{j,m}$, $P^m_j$ (through $G^m_j$) and the spaces $\Vee^{m'}_{j'}$ with $j'<j$ and $m'$ arbitrary, or $j'=j$ and $m'<m$. The relation \eqref{BasicLemma4} follows from \eqref{BasicLemma3} and induction, and then \eqref{BasicLemma5} (where $P^m_j$ actually means its restriction to $\W_{j,m+1}$) is a definition by induction; it clearly gives that the $P^m_j(t)$ have values in $\W_j'$ as required in \eqref{BasicLemma0}.

\medskip

We will illustrate the lemma and its proof with an example and then give a proof.

\begin{example}
Suppose $\tilde \Sing$ is spanned by elements $e_{j,k}$ ($j=0,1$ and $k=1,\dotsc,K_j)$ and that the Hermitian inner product is defined so that these vectors are orthonormal. Define the linear operator $\tilde N:\tilde \Sing\to \tilde \Sing$ so that $\tilde N e_{j,1}=0$ and $\tilde N e_{j,k}=e_{j,k-1}$ for $1<k\leq K_j$. Thus $\tilde N^k e_{j,k}=0$ and $\tilde N^{k} e_{j,k+1}=e_{j,1}\ne 0$. Pick integers $0\leq s_0<s_1<\min\set{K_0,K_1}$, and let
\begin{equation*}
\W=\LinSpan\set{e_{0,s_0+1}, e_{1,s_1+1}, e_{0,s_1+1}+e_{1,s_1}}.
\end{equation*}
If $w\in \W$ and $w\ne 0$, then $e^{t\tilde N}w$ is a polynomial of degree exactly $s_0$ or $s_1$. Let $\W_{0,0}=\W\cap \ker \tilde N^{s_0+1}$, \ie,
\begin{equation*}
\W_{0,0}=\LinSpan\set{e_{0,s_0+1}}.
\end{equation*}
Then $e^{t\tilde N}w$ is  polynomial of degree $s_0$ if $w\in \W_{0,0}$. Let $\W_{1,0}=\W\cap \ker \tilde N^{s_1+1}\cap \W_{0,0}^{\perp}$. Thus
\begin{equation*}
\W_{1,0}=\LinSpan\set{e_{1,s_1+1},e_{0,s_1+1}+e_{1,s_1}}
\end{equation*}
and $e^{t\tilde N}w$ is  polynomial of degree exactly $s_1$ if $w\in \W_{1,0}$ and $w\ne 0$. With these spaces we have $\W=\W_{0,0}\oplus\W_{1,1}$ as an orthogonal sum. By \eqref{BasicLemmaItem1} of Lemma \ref{BasicLemma}, $P^0_0=\id_{\W_{0,0}}$. So $e^{t\tilde N}P^0_0$ is the restriction of
\begin{equation*}
e^{t\tilde N}=\sum_{k=0}^{s_0}\frac{t^k}{k!}\tilde N^k
\end{equation*}
to $\W_{0,0}$, $n^0_0=s_0$, and $G^0_0$ is $\frac{1}{s_0!}\tilde N^{s_0}$ restricted to $\W_{0,0}$. Thus $\Vee^0_0=\LinSpan\set{e_{0,1}}$. The space $\W_{0,1}$, defined using \eqref{BasicLemma3}, is the zero space by the convention on sums where the upper index is less than the lower index. Thus $M_0=0$. We next analyze what the lemma says when $j=1$. As when $j=0$, $P^0_1=\id_{\W_{1,0}}$, so $e^{t\tilde N}P^0_1$ is the restriction of 
\begin{equation*}
e^{t\tilde N}=\sum_{k=0}^{s_1}\frac{t^k}{k!}\tilde N^k
\end{equation*}
to $\W_{1,0}$. Hence $n^0_1=s_1$, and $G^0_1=\frac{1}{s_1!}\tilde N^{s_1}|_{\W_{1,0}}$. The preimage of $\Vee^0_0$ by $G^0_1$ is $\W_{1,1}=\LinSpan\set{e_{0,s_1+1}+e_{1,s_1}}$, and so $\W^0_1=\LinSpan\set{e_{1,s_1+1}}$ and $\Vee^0_1=\LinSpan\set{e_{1,1}}$. With $w=e_{0,s_1+1}+e_{1,s_1}$ we have 
\begin{equation*}
G^0_1w=\frac{1}{s_1!}e_{0,1}=G^0_0 \frac{s_0!}{s_1!}e_{0,s_0+1},
\end{equation*}
so with $F^{0,1}_{0,1}:\W_{1,1}\to \W^0_0$ defined by
\begin{equation*}
F^{0,1}_{0,1}w=-\frac{s_0!}{s_1!}e_{0,s_0+1}
\end{equation*}
we have $G^0_1+G^0_0F^{0,1}_{0,1}=0$. Formula \eqref{BasicLemma5} reads
\begin{equation*}
P^1_1(t) = \id_{\W_{1,1}}+t^{s_1-s_0}F^{0,1}_{0,1}
\end{equation*}
in this instance, and
\begin{equation*}
e^{t\tilde N} P^1_1(t) w= \sum_{k=0}^{s_1} \frac{t^k}{k!}\tilde N^k w - \frac{s_0!\,t^{s_1-s_0}}{s_1!} \sum_{k=0}^{s_0}\frac{t^k}{k!} \tilde N^k e_{0,s_0+1}.
\end{equation*}
In the first sum the highest order term is $t^{s_1}\!/s_1!\, e_{0,1}$, while in the second it is $t^{s_0}\!/s_0!\, e_{0,1}$. Taking into account the coefficient of the second sum we see that $e^{t\tilde N} P^1_1(t) w$ has order $<s_1$. A more detailed calculation gives that the order is $s_1-1$, and that the leading coefficient is given by the map
\begin{equation*}
w\mapsto \big( \frac{1}{(s_1-1)!}-\frac{s_0!}{s_1!(s_0-1)!}\big)e_{0,2} + e_{1,1};
\end{equation*}
its image spans $\Vee^1_1$. Note that $\Vee^0_0+\Vee_1^0+\Vee^1_1$ is a direct sum and is invariant under $\tilde N$.
\end{example}

\begin{proof}[Proof of Lemma \ref{BasicLemma}]
We note first that the properties of the objects in the lemma are such that 
\begin{equation}\label{DirectSum}
D_{\mu,\infty}=\sum_{j=0}^J\sum_{m=0}^{M_{j}}\Vee^{m}_{j}
\end{equation}
is a direct sum. Indeed, suppose we have $w^m_{j}\in \W^m_j$, $j=0,\dotsc,J$, $m=0,\dotsc,M_j$ such that
\begin{equation*}
\sum_{j=0}^J\sum_{m=0}^{M_{j}}G^{m}_{j}w^m_{j}=0.
\end{equation*}
If some $w^m_j\ne 0$, let 
\begin{equation*}
j_0=\max\set{j:\exists m \text{ s.t. } w^m_j\ne 0},\quad m_0=\max\set{m:w^m_{j_0}\ne 0},
\end{equation*}
so that $w^{m_0}_{j_0}\ne 0$. Thus 
\begin{equation*}
G^{m_0}_{j_0}w^{m_0}_{j_0} = -\sum_{j=0}^{j_0-1}\sum_{m=0}^{M_{j}}G^{m}_{j}w^{m}_{j} - \sum_{m=0}^{m_0-1} G^{m}_{j_0}w^{m}_{j_0}\in \sum_{j=0}^{j_0-1}\sum_{m=0}^{M_{j}}\Vee ^{m}_{j} - \sum_{m=0}^{m_0-1} \Vee^{m}_{j_0},
\end{equation*}
therefore $w^{m_0}_{j_0}\in \W_{j_0,m_0+1}$ by \eqref{BasicLemma3}. But also $w^{m_0}_{j_0}\in \W^{m_0}_{j_0}$, a space which by definition is orthogonal to $\W_{j_0,m_0+1}$. Consequently $w^{m_0}_{j_0}=0$, a contradiction. It follows that \eqref{DirectSum} is a direct sum as claimed, and in particular that the maps
\begin{equation*}
G^m_j|_{\W^m_j}:\W^m_j\to \Vee^m_j
\end{equation*}
are isomorphisms.

Note that $e^{t \tilde N}w$ is a nonzero polynomial whenever $w\in \W\minus 0$ and let
\begin{equation*}
\set{s_j}_{j=0}^J=\set{\deg e^{t \tilde N}w:w\in \W,\ w\ne 0}
\end{equation*}
be an enumeration of the degrees of these polynomials, in increasing order. Let $\W_{-1,0}=\set{0}\subset \W$ and inductively define
\begin{equation*}
\W_{j,0}=\W\cap \ker \tilde N^{s_j+1}\cap \W_{j-1,0}^\perp, \quad j=0,\dotsc, J.
\end{equation*}
Thus $\W_{j,0}\subset \W$, $\W=\bigoplus_{j=0}^J \W_{j,0}$ is an orthogonal decomposition of $\W$,
\begin{equation*}
\tilde N^{s_j}|_{\W_{j,0}}:\W_{j,0}\to \tilde \Sing
\end{equation*}
is injective for $j=0,\dotsc,J$, and if $w\in \W_{j,0}\minus 0$ then $e^{t\tilde N}w$ is a polynomial of degree exactly $s_j$. The spaces $\W^m_j$ will be defined so that $\bigoplus_m\W^m_j = \W_{j,0}$.

Let $P^0_0(t)=\id_{\W_{0,0}}$, let $Q^0_0(t)=e^{t\tilde N}P^0_0(t)$. Then  $\ord (Q^0_0)=s_0$ and
\begin{equation*}
G^0_0=1/s_0!\,\tilde N^{s_0}|_{\W_{0,0}}.
\end{equation*}
By \eqref{BasicLemma3}, $\W_{0,1}$ is the preimage of the zero vector space. Since $\tilde N^{s_0}$ is injective on $\W_{0,0}$, $\W_{0,1}=0$, $\W^0_0=\W_{0,0}$ and $M_0=0$. Let $\Vee^0_0=G^0_0(\W^0_0)$. This proves the lemma if $J=0$.

We continue the proof using induction on $J$. Suppose that $J\geq 1$ and that the lemma has been proved for $\W'=\bigoplus_{j=0}^{J-1}\W_{j,0}$, so we have all objects described in the statement of the lemma, for $\W'$. The corresponding objects for $\W_{J,0}$ are then defined by induction in the second index, as follows.

First, let $P^0_J(t)=\id_{\W_{J,0}}$, $Q^0_J=e^{t\tilde N}P^0_J$ (a polynomial in $t$ of degree $n^0_J=s_J$) and $G^0_J=\coeff_{s_J}(Q^0_J)$.

Next, suppose we have found $\W_{J,0}\supset\dots\supset\W_{J,M-1}$ and 
\begin{equation*}
P^{m}_j\in L(\W_{j,m},\W)[t,t^{-1}]
\end{equation*}
so that the properties described in the lemma are satisfied for $j<J$ and all $m$, or $j=J$ and $m\leq M-1$. As  discussed already it follows that 
\begin{equation*}
\sum_{m=0}^{M-2} \Vee^m_J + \sum_{j=0}^{J-1} \sum_{m=0}^{M_j}\Vee^m_j
\end{equation*}
is a direct sum and that the maps
\begin{equation}\label{RestrictedGmj}
G^m_j|_{\W^m_j}:\W^m_j\to \Vee^m_j
\end{equation}
defined so far are isomorphisms. Suppose further that the $n^m_J=\ord(Q^m_J)$, $m=0,\dotsc,M-1$ are nonnegative and strictly decrease as $m$ increases. In agreement with \eqref{BasicLemma3}, let
\begin{equation*}
\W_{J,M}=(G^{M-1}_J)^{-1} (\sum_{m=0}^{M-2} \Vee^m_J + \sum_{j=0}^{J-1} \sum_{m=0}^{M_j}\Vee^m_j),
\end{equation*}
a subspace of the domain $\W_{J,M-1}$ of $G^{M-1}_J$. Define $\W^{M-1}_J=\W_{J,M-1}\cap \W_{J,M}^\perp$. If $w\in \W_{J,M}$, then
\begin{equation*}
G^{M-1}_J w=\sum_{m=0}^{M-2} v^m_J + \sum_{j=0}^{J-1} \sum_{m=0}^{M_j}v^m_j
\end{equation*}
uniquely with $v^m_j\in \Vee^m_j$. Since the maps \eqref{RestrictedGmj} are isomorphisms, there are unique maps $F^{m,M}_{j,J}:\W_{J,M}\to \W^m_j$, $j=0,\dotsc,J-1$ and $m=0,\dotsc,M_j$, or $j=J$ and $m=0,\dotsc,M-2$ such that 
\begin{equation*}
G^{M-1}_J+\sum_{m=0}^{M-2} G^m_JF^{m,M}_{J,J} + \sum_{j=0}^{J-1} \sum_{m=0}^{M_j}G^m_jF^{m,M}_{j,J}=0
\end{equation*}
on $\W_{J,M}$, that is, \eqref{BasicLemma4} holds. Define
\begin{equation*}
P^M_J=P^{M-1}_J+\sum_{m=0}^{M-2} t^{n^{M-1}_J-n^m_J}P^m_J F^{m,M}_{J,J} + \sum_{j=0}^{J-1} \sum_{m=0}^{M_j}t^{n^{M-1}_J-n^m_{j}}P^m_jF^{m,M}_{j,J}
\end{equation*}
so \eqref{BasicLemma5} holds. Let $Q^M_J=e^{t\tilde N}P^M_J$. Because of \eqref{BasicLemma2}, each term on the right in
\begin{equation*}
Q^M_J=Q^{M-1}_J+\sum_{m=0}^{M-2} t^{n^{M-1}_J-n^m_J}Q^m_J F^{m,M}_{J,J} + \sum_{j=0}^{J-1} \sum_{m=0}^{M_j}t^{n^{M-1}_J-n^m_{j}}Q^m_jF^{m,M}_{j,J}.
\end{equation*}
has order $n^{M-1}_J$, so $\coeff_n(Q^M_J)=0$ if $n\geq n^{M-1}_J$. If $Q^M_J\ne 0$, let $n^M_J=\ord(Q^M_J)$. {\it A fortiori} $n^M_J<n^{M-1}_J$.

We now show that if $Q^M_J=0$, then $\W_{J,M}=0$, so $M_J=M-1$ and the inductive construction stops.

Let $F^{m,m+1}_{j,j}:\W_{j,m+1}\to \W_{j,m}$ be the inclusion map. Note that the combination of indices just used does not appear in \eqref{BasicLemma4}: these maps are not defined in the statement of the lemma. With this notation 
\begin{equation}\label{mthRat}
P^m_J=\sum_{m'=0}^{m-1} t^{n^{m-1}_J-n^{m'}_J}P^{m'}_JF^{m',m}_{J,J} + \tilde H^m_J
\end{equation}
for $m=1,\dotsc,M$ and some $\tilde H^m_J\in L(\W_{J,m},\W')[t,t^{-1}]$. Let $\mathcal P_m$ be the set of finite strictly increasing sequences $\pmb \nu=(\nu_0,\nu_1,\dotsc,\nu_k)$ of elements of $\set{0,\dotsc,m}$ with $\nu_0=0$ and $\nu_k=m$. For $\nu=(\nu_0,\dotsc,\nu_k)\in \mathcal P_m$ ($m\geq 1$) define
\begin{gather*}
F^{\pmb \nu}_J=F^{\nu_0,\nu_1}_{J,J}\circ \dots \circ F^{\nu_{m-1},\nu_m}_{J,J}\\
n^{\pmb \nu}_J=(n^{\nu_1-1}_J-n^{\nu_0}_J)+(n^{\nu_2-1}_J-n^{\nu_1}_J)+\dots + (n^{\nu_k-1}_J-n^{\nu_{k-1}}_J).
\end{gather*}
Since the $n^{m'}_J$ strictly decrease as $m'$ increases, the numbers $n^{\pmb \nu}_J$ are strictly negative except when $\pmb\nu$ is the maximal sequence $\pmb \nu_{\max}$ in $\set{0,\dotsc,m}$, in which case $n^{\pmb\nu_{\max}}_J=0$ and $F^{\pmb\nu_{\max}}$ is the inclusion of $\W_{J,m}$ in $\W_{J,0}$. It is not hard to prove (by induction on $m$, using \eqref{mthRat}) that
\begin{equation}\label{mthRatbis}
P^m_J=P^0_J\sum_{\pmb \nu \in \mathcal P_m} t^{n^{\pmb \nu}_J} F^{\pmb \nu}_J +H^m_J
\end{equation}
for all $m\geq 1$ where $H^m_J\in L(\W_{J,m},\W')[t,t^{-1}]$. If $Q^M_J=0$, then $P^M_J=0$, so, since $\tilde N^{s_J} H^M_J=0$, 
\begin{equation*}
\tilde N^{s_J}P^M_J=\sum_{\pmb \nu \in \mathcal P_M} t^{n^{\pmb \nu}_J} \tilde N^{s_J}F^{\pmb \nu}_J=0
\end{equation*}
In particular, $\tilde N^{s_J} F^{\pmb \nu_{\max}}_J=\coeff_{0}(\tilde N^{s_J}P^M_J)=0$. Since $\tilde N^{s_J}$ is injective on $\W_{J,0}$, we conclude that the inclusion of $\W_{J,M}$ in $\W_{J,0}$ is zero. This means that $\W_{J,M}=0$, so the inductive construction stops with $M_J=M-1$. 

We will now show that there is a finite $M$ such that $Q^M_J=0$. The inductive construction gives, as long as $Q^m_J\ne 0$, the numbers $n^m_J=\ord(Q^m_J)$ which form a strictly decreasing sequence in $m$, with $n^0_J=s_J$. Suppose $n^{M-1}_J\geq 0$, $Q^M_J\ne 0$, and $n^M_J<0$. In particular, the coefficient of $t^0$ in $Q^M_J$ vanishes. Using \eqref{mthRatbis} with $m=M$ we have
\begin{equation*}
e^{t\tilde N} P^M_J = \sum_{\pmb \nu \in \mathcal P_M}\sum_{s=0}^{s_J}\frac{t^{s+n^{\pmb \nu}_J}}{s!} \tilde N^sF^{\pmb \nu}_J +e^{t\tilde N}H^M_J
\end{equation*}
The coefficient of $t^0$ is
\begin{equation*}
\coeff_0(\e^{t\tilde N} P^M_J) = \sum_{\pmb \nu \in \mathcal P_M}\frac{1}{(-n^{\pmb \nu}_J)!} \tilde N^{-n^{\pmb \nu}_J}F^{\pmb \nu}_J +\coeff_0(e^{t\tilde N}H^M_J);
\end{equation*}
recall that $n^{\pmb \nu}_J\leq 0$. Since $H^M_J$ maps into $\W'$, $\tilde N^{s_J}\coeff_0(H^M_J)=0$, and since $\tilde N^{s}|_{\W_{J,0}}=0$ if $s>s_J$, $\tilde N ^{s_J}\tilde N^{-n^{\pmb \nu}_J}=0$ if $n^{\pmb \nu}_J\ne 0$. Thus
\begin{equation*}
\tilde N^{s_J}\coeff_0(\e^{t\tilde N} P^M_J) =  \tilde N^{s_J}F^{\pmb \nu_{\max}}_J
\end{equation*}
where $\pmb \nu_{\max}=(0,1,\dotsc,M)$. Since $c_0(e^{t\tilde N}P^M_J)=0$ by hypothesis, since $F^{\pmb\nu_{\max}}_J$ is the inclusion of $\W_{J,M}$ in $\W_{J,0}$, and since $\tilde N^{s_J}$ is injective on $\W_{J,0}$, $\W_{J,M}=0$.
\end{proof}

\begin{proof}[Proof of Lemma \ref{ExistenceOfShadowLemma}]
Apply Lemma \ref{BasicLemma} to each of the spaces $\W_{\mu_\ell} = \tilde \pi_{\mu_\ell}D_{\mu_\ell}$. The corresponding objects are labeled adjoining $\ell$ as a subindex. Get in particular, decompositions
\begin{equation*}
\tilde \pi_{\mu_\ell}D_{\mu_\ell} = \bigoplus_{j=0}^{J_\ell}\bigoplus_{m=0}^{M_{j,\ell}}\W^m_{j,\ell}\subset \tilde\Sing_{\mu_\ell}
\end{equation*}
for each $\ell$, and operators $G^m_{j,\ell}:\W^m_{j,\ell}\to \Vee^m_{j,\ell}\subset \tilde \Sing_{\mu_\ell}$ such that
\begin{equation*}
\bigoplus_{j=0}^{J_\ell}\bigoplus_{m=0}^{M_{j,\ell}} G^m_{j,\ell}|_{\W^m_{j,\ell}}: \bigoplus_{j=0}^{J_\ell}\bigoplus_{m=0}^{M_{j,\ell}}\W^m_{j,\ell}\to  D_{\mu_\ell,\infty}=\bigoplus_{j=0}^{J_\ell}\bigoplus_{m=0}^{M_{j,\ell}} \Vee^m_{j,\ell} 
\end{equation*}
is an isomorphism. Let $d^m_{j,\ell}=\dim \W^m_{j,\ell}$ and pick a basis
\begin{equation*}
w^m_{j,\ell.k},\quad  1\leq k\leq d^m_{j,\ell}
\end{equation*}
of $\W^m_{j,\ell}$, $j=0,\dotsc,J_\ell$, $m=0,\dotsc,M_{j,\ell}$. Then $\tilde p^m_{j,\ell,k}(t)=t^{-n^m_{j,\ell}}P^m_{j,\ell}(t)w^m_{j,\ell,k}\in \W_{\mu_\ell}$. These elements
\begin{equation*}
\tilde p^m_{j,\ell,k}\in \W_{\mu_\ell}[t,t^{-1}], \quad j=0,\dotsc,J_\ell,\ m=0,\dotsc,M_{j,\ell},\ \ell=1,\dotsc d^m_{j,\ell}
\end{equation*}
are the ones Lemma \ref{ExistenceOfShadowLemma} claims exist. Indeed, since $Q^m_{j,\ell}(t)=e^{t\tilde N_{\mu_\ell}}P^m_{j,\ell}(t)$,
\begin{equation*}
\lim_{\substack{\Re t\to \infty\\t\in \Strip_\theta}}  e^{t\tilde N_{\mu_\ell}} t^{-n^m_{j,\ell}} P^m_{j,\ell}(t) w^m_{j,\ell,k} = G^m_{j,\ell}w^m_{j,\ell,l}.
\end{equation*}
Since the $G^m_{j,\ell}w^m_{j,\ell,k}$ form a basis of $D_{\mu,\infty}$, the  $t^{-n^m_{j,\ell}}P^m_{j,\ell}(t)w^m_{j,\ell,k}$, form a basis of $\W_{\mu_\ell}$ for all $t\in \Strip_\theta$ with large enough real part.
\end{proof}

\section{Asymptotics of the projection}
\label{sec:AsymptoticsOfProjection}

With the setup and (slightly changed) notation leading to and in the proof of Proposition \ref{ExistenceOfShadow}, given a subspace $D\subset \Sing$ and the linear map $\a:\Sing\to\Sing$ we have, for fixed $\theta\geq 0$ and $t\in \Strip_\theta=\smallset{t\in \C: |\Im t|\leq \theta}$, that
\begin{equation*}
e^{t\a}D = \LinSpan \set{v_k(t)},\quad \Re t\gg 0
\end{equation*}
with
\begin{equation}\label{TheVkAgain}
v_k(t)=e^{t\a'} g_k(t) +\sum_{\substack{\lambda\in \spec \a\\\Re\lambda<\mu_k}} e^{t(\lambda -\mu_k)}\hat p_{k,\lambda}(t).
\end{equation}
The $g_k(t)$ are polynomials in $1/t$ with values in $\tilde\Sing_{\mu_k}$, the collection of vectors
\begin{equation*}
g_{\infty,k}=\lim_{t\to\infty}g_k(t)
\end{equation*}
is a basis of $D_\infty$, the $\mu_k$ form a finite sequence, possibly with repetitions, of elements in the set $\set{\Re\lambda :\lambda\in \spec\a}$, and
\begin{equation*}
\hat p_{k,\lambda}(t)=e^{tN_\lambda}\pi_\lambda p_k(t)
\end{equation*}
where the $p(t)$ are polynomials in $t$ and $1/t$ with values in $\Sing$. The additive semigroup $\mathfrak S_\a\subset \C$ (possibly without identity) generated by the set \eqref{GeneratorOfSemiGroup-a} is a subset of $\set{\lambda\in \C:\Re \lambda\leq 0}$ and has the property that $\smallset{\vartheta\in \mathfrak S_\a:\Re \vartheta> \mu}$ is finite for every $\mu\in \R$.

\begin{proposition}\label{AsymptOfProjA}
Let $K\in \Gr_{d'}(\Sing)$ be complementary to $D$, suppose that
\begin{equation}\label{HypA}
\mathscr V_K\cap \Omega^+_\theta(D) = \emptyset.
\end{equation}
Then there are polynomials $p_{\vartheta}(z^1,\dotsc,z^N,t)$ with values in $\End(\Sing)$ and $\C$-valued polynomials $q_{\vartheta}(z^1,\dotsc,z^N,t)$ such that 
\begin{equation}\label{qBddBelowA}
\exists\, C,\ R_0>0 \text{ such that }|q_\vartheta(e^{\im t\Im\lambda_1},\dotsc,e^{\im t\Im\lambda_N},t)|>C\text{ if }t\in \Strip_\theta,\ \Re t>R_0
\end{equation}
and such that
\begin{equation*}
\pi_{e^{t\a}D,K} = \sum_{\vartheta\in \mathfrak S_\a} \frac{e^{t\vartheta}p_\vartheta(e^{\im t\Im\lambda_1},\dotsc,e^{\im t\Im\lambda_N},t)}{q_{\vartheta}(e^{\im t\Im\lambda_1},\dotsc,e^{\im t\Im\lambda_N},t)},\quad t\in \Strip_\theta,\ \Re t>R_0
\end{equation*}
with uniform convergence in norm in the indicated subset of $\Strip_\theta$.
\end{proposition}

\begin{proof}
Let $K\subset \Sing$ be complementary to $D$ as indicated in the statement of the proposition, let $\mathbf u=[u_1,\dotsc,u_{d'}]$ be an ordered basis of $K$. Write $\mathbf g$ for an ordering of the basis $\set{g_{\infty,k}}$ of $D_\infty$. With the $v_k(t)$ ordered as the $g_{\infty,k}$ to form $\mathbf v(t)$, we have
\begin{equation}\label{Bases}
\begin{bmatrix} \mathbf v(t) & \mathbf u\end{bmatrix}
=\begin{bmatrix} \mathbf g & \mathbf u\end{bmatrix} \cdot
\begin{bmatrix} \alpha(t)&0 \\\beta(t) &\id\end{bmatrix}
\end{equation}
where
\begin{equation}\label{AlphaAndBeta}
\alpha(t) = \sum_k\sum_{\substack{\lambda\in \spec a\\\Re\lambda\leq\mu_k}}e^{t(\lambda-\mu_k)}\alpha_{k,\lambda}(t),\qquad
\beta(t) = \sum_k\sum_{\substack{\lambda\in \spec a\\\Re\lambda\leq\mu_k}}e^{t(\lambda-\mu_k)}\beta_{k,\lambda}(t).
\end{equation}
The entries of the matrices $\alpha_{k,\lambda}(t)$ and $\beta_{k,\lambda}(t)$ are both polynomials in $t$ and $1/t$, but only in $1/t$ if $\Re\lambda=\mu_k$.  Define
\begin{equation}\label{alpha0andtilde}
\alpha^{(0)}(t)=\sum_k\sum_{\substack{\lambda\in \spec a\\\Re\lambda=\mu_k}}e^{t(\lambda-\mu_k)}\alpha_{k,\lambda}(t),\qquad
\tilde \alpha(t) = \sum_k\sum_{\substack{\lambda\in \spec a\\\Re\lambda<\mu_k}}e^{t(\lambda-\mu_k)} \alpha_{k,\lambda}(t),
\end{equation}
likewise $\beta^{(0)}(t)$ and $\tilde \beta(t)$. Note that $\tilde \alpha(t)$ and $\tilde \beta(t)$ decrease exponentially as $\Re t\to\infty$ with $|\Im t|$ bounded.

The hypothesis \eqref{HypA} gives that
\begin{equation*}
\begin{bmatrix} \alpha(t)&0 \\\beta(t) &\id\end{bmatrix}
\end{equation*}
is invertible for every sufficiently large $\Re t$, so $\alpha(t)$ is invertible for such $t$. In fact,
\begin{equation}\label{DetBddBelow}
\text{there are } C,\ R_0>0 \text{ such that }|\det(\alpha(t))|>C\text{ if }t\in \Strip_\theta,\ \Re t>R_0.
\end{equation}
For suppose this is not the case. Then there is a sequence $\set{t_\nu}$ in $\Strip_\theta$ with $\Re t_\nu\to \infty$ as $\nu\to\infty$ such that $\det \alpha(t_\nu)\to 0$. Since both $\alpha(t_\nu)$ and $\beta(t_\nu)$ are bounded, we may assume, passing to a subsequence, that they converge. It follows that $e^{t_\nu a}D$ converges, by definition, to an element $D'\in \Omega^+_\theta(D)$. Also the matrix in \eqref{Bases} converges. The vanishing of the determinant of the limiting matrix implies that $K\cap D'\ne \set{0}$, contradicting \eqref{HypA}. Thus \eqref{DetBddBelow} holds.

If $\phi\in \Sing$ then of course
\begin{equation*}
\phi = \begin{bmatrix} \mathbf g& \mathbf u \end{bmatrix} \cdot \begin{bmatrix} \varphi^1\\ \varphi^2 \end{bmatrix}
\end{equation*}
where the $\varphi^i$ are columns of scalars. Substituting
\begin{equation*}
\begin{bmatrix}\mathbf g & \mathbf u\end{bmatrix} = \begin{bmatrix}\mathbf v(t) & \mathbf u\end{bmatrix} \cdot \begin{bmatrix} \alpha(t)^{-1}&0 \\ -\beta(t)\alpha(t)^{-1} &\id\end{bmatrix}
\end{equation*}
gives
\begin{equation*}
\phi = \begin{bmatrix}\mathbf v(t) & \mathbf u\end{bmatrix} \cdot \begin{bmatrix} \alpha(t)^{-1}&0 \\ -\beta(t)\alpha(t)^{-1} &\id\end{bmatrix} \begin{bmatrix} \varphi^1\\ \varphi^2 \end{bmatrix} = \begin{bmatrix}\mathbf v(t) & \mathbf u\end{bmatrix} \cdot \begin{bmatrix} \alpha(t)^{-1}\varphi^1\\ -\beta(t)\alpha(t)^{-1}\varphi^1 +\varphi^2\end{bmatrix},
\end{equation*}
hence
\begin{equation*}
\phi = \mathbf v(t) \cdot \alpha(t)^{-1}\varphi^1+ 
\mathbf u  \cdot\big(-\beta(t)\alpha(t)^{-1}\varphi^1+\varphi^2\big);
\end{equation*}
This is the decomposition of $\phi$ according to $\Sing=e^{t\a}D\oplus K$, therefore
\begin{equation*}
\pi_{e^{t\a}D,K}\phi = 
\mathbf v(t) \cdot \alpha(t)^{-1}\varphi^1.
\end{equation*}
Replacing $\mathbf v(t) = \mathbf g \cdot \alpha(t)+\mathbf u \cdot \beta(t)$ we obtain
\begin{equation}\label{ProjInCoords}
\begin{aligned}
\pi_{e^{t\a}D,K}\phi &= 
\big(\mathbf g \cdot \alpha(t)+\mathbf u \cdot \beta(t)\big) \alpha(t)^{-1}\varphi^1 \\&= \big(\mathbf g+\mathbf u \cdot \beta(t)\alpha(t)^{-1}\big)\varphi^1.
\end{aligned}
\end{equation}
The matrix $\alpha^{(0)}(t)$ is invertible because of \eqref{DetBddBelow} and the decomposition $\alpha(t)=\alpha^{(0)}(t)+\tilde \alpha(t)$, so
\begin{equation}\label{betaalphaInverse}
\begin{aligned}
\beta(t)\alpha(t)^{-1}
&=\beta(t)\alpha^{(0)}(t)^{-1}\big(\id + \tilde\alpha(t)\alpha^{(0)}(t)^{-1}\big)^{-1}\\
&=\beta(t)\alpha^{(0)}(t)^{-1}\sum_{\ell=0}^\infty(-1)^\ell [\tilde\alpha(t)\alpha^{(0)}(t)^{-1}]^\ell.
\end{aligned}
\end{equation}
The series converges absolutely and uniformly in $\set{t\in \Strip_\theta:\Re t > R_0}$ for some real $R_0\in\R$. The entries of $\alpha^{(0)}(t)$ are expressions
\begin{equation*}
\sum_{\lambda\in \spec\a}e^{\im t\Im\lambda}\sum _{\nu=0}^N  c_{\lambda,\nu}t^{-\nu},
\end{equation*}
hence
\begin{equation*}
\det\alpha^{(0)}(t)=q(e^{\im t\Im\lambda_1},\dotsc,e^{\im t\Im\lambda_N},1/t)
\end{equation*}
for some polynomial $q(z^1,\dotsc,z^N,1/t)$. Note that because of \eqref{DetBddBelow},
\begin{equation}\label{Det0BddBelow}
\text{there are } C,\ R_0>0 \text{ such that }|\det(\alpha^0(t))|>C\text{ if }t\in \Strip_\theta,\ \Re t>R_0.
\end{equation}
Since $\alpha^{(0)}(t)^{-1}=(\det\alpha^{(0)}(t))^{-1}\Delta(t)^{\dag}$ where $\Delta(t)^{\dag}$ is the matrix of cofactors of $\alpha^{(0)}(t)$, \eqref{betaalphaInverse} and \eqref{AlphaAndBeta} give 
\begin{equation}\label{ProtoExpansion}
\beta(t)\alpha(t)^{-1} = \sum_{\vartheta\in \mathfrak S_\a} r_{\vartheta}(t)e^{t\vartheta} 
\end{equation}
where $\mathfrak S_\a$ was defined before the statement of Proposition \ref{AsymptOfProjA} as the additive semigroup generated by $\set{\lambda - \Re\lambda':\lambda,\lambda'\in \spec\a,\ \Re\lambda\leq \Re\lambda'}$ and $r_{\vartheta}(t)$ is a matrix whose entries are of the form
\begin{equation*}
\frac{p_\vartheta(e^{\im t\Im\lambda_1},\dotsc,e^{\im t\Im\lambda_N},t,1/t)}{q(e^{\im t\Im\lambda_1},\dotsc,e^{\im t\Im\lambda_N},1/t)^{n_\vartheta}}
\end{equation*}
for some polynomial $p_\vartheta(z^1,\dotsc,z^N,t,1/t)$ and nonnegative integers $n_\vartheta$. Multiplying the numerator and denominator by the same nonnegative (integral) power of $t$ we replace the dependence on $1/t$ by polynomial dependence in $e^{\im t\Im\lambda_1},\dotsc,e^{\im t\Im\lambda_N},t$ only. This gives the structure of the ``coefficients'' of the $e^{t\vartheta}$ stated in the proposition for the expansion of $\pi_{e^{t\a}D,K}$.
\end{proof}

The terms in \eqref{ProtoExpansion} with $\Re\vartheta=0$ come from $\beta^{(0)}(t)\alpha^{(0)}(t)^{-1}$. So the principal part of $\pi_{e^{t\a}D,K}$ is
\begin{equation*}
\sym(\pi_{e^{t\a}D,K})\phi=\big(\mathbf g+\mathbf u \cdot \beta^{(0)}(t)\alpha^{(0)}(t)^{-1}\big)\varphi^1
\end{equation*}
This principal part is not itself a projection, but 
\begin{equation*}
\|\sym(\pi_{e^{t\a}D,K})-\pi_{e^{t\a'}D_\infty,K}\|\to 0\quad \text{ as }\Re t\to \infty,\ t\in \Strip_\theta.
\end{equation*}

We now restate Proposition \ref{AsymptOfProjA} as an asymptotics for the family \eqref{PiScaling} using the notation $\bkappa$ for the action on $\Sing$ and express the asymptotics of $\pi_{\bkappa^{-1}_{\zeta^{1/m}} D,K}$ in terms of the boundary spectrum of $A$ exploiting \eqref{TheSpectrum}. Condition \eqref{HypB} below corresponds to our geometric condition in part (iii) of Theorem~\ref{EquivRayConditionsWedge} expressing the fact that $\Lambda$ is a sector of minimal growth for $A_{\wedge,\Dom_\wedge}$. The $\Omega$-limit set is the one defined in \eqref{OmegaMinus}.
 Recall that by $\zeta^{1/m}$ we mean the root defined by the principal branch of the logarithm on $\C\minus \overline \R_-$. We let $\lambda_0\ne 0$ be an element in the central axis of $\Lambda$ and define $\tilde \Lambda=\set{\zeta:\zeta\lambda_0\in \Lambda}$; this is a closed sector not containing the negative real axis.

Let $\mathfrak S\subset \C$ be the additive semigroup generated by 
\begin{equation*}
\set{\sigma-\im \Im\sigma':\sigma,\sigma'\in \spec_b(A),\ -m/2<\Im \sigma \leq \Im \sigma'<m/2}.
\end{equation*}
Thus $-\im\mathfrak S=\mathfrak S_\a$. Let $\sigma_1,\dotsc,\sigma_N$ be an enumeration of the elements of $\Sigma=\spec_b(A)\cap\set{-m/2<\Im\sigma<m/2}$. 

\begin{theorem}\label{AsymptOfProjB}
Let $K\in \Gr_{d'}(\Sing)$ be complementary to $D$, suppose that
\begin{equation}\label{HypB}
\mathscr V_K\cap \Omega^-_\Lambda(D) = \emptyset.
\end{equation}
Then there are polynomials $p_{\vartheta}(z^1,\dotsc,z^N,t)$ with values in $\End(\Sing)$ and $\C$-valued polynomials $q_{\vartheta}(z^1,\dotsc,z^N,t)$ such that 
\begin{equation}\label{qBddBelowB}
\exists\, C,\ R_0>0 \text{ such that }|q_\vartheta(\zeta^{\im \Re\sigma_1/m},\dotsc,\zeta^{\im \Re\sigma_N/m},t)|>C\text{ if } \zeta\in \tilde \Lambda,\ |\zeta|>R_0
\end{equation}
and such that
\begin{equation*}
\pi_{\bkappa^{-1}_{\zeta^{1/m}D},K} = \sum_{\vartheta\in \mathfrak S} \frac{\zeta^{-\im \vartheta/m}p_\vartheta(\zeta^{\im \Re\sigma_1/m},\dotsc,\zeta^{\im \Re\sigma_N/m},m^{-1} \log \zeta)}{q_{\vartheta}(\zeta^{\im \Re\sigma_1/m},\dotsc,\zeta^{\im \Re\sigma_N/m},m^{-1}\log\zeta)},\ \zeta\in \tilde \Lambda,\ |\zeta|>R_0
\end{equation*}
with uniform convergence in norm in the indicated subset of $\tilde\Lambda$.
\end{theorem}

The elements $\vartheta\in \mathfrak S$ are of course finite sums $\vartheta = \sum n_{jk}(\sigma_j-\im \Im\sigma_k)$ for some nonnegative integers $n_{jk}$, with $\sigma_j$, $\sigma_k\in \Sigma$ and $\Im \sigma_j\leq  \Im\sigma_k$. Separating real and imaginary parts we may write $\zeta^{-\im \vartheta/m}$ as a product of factors
\begin{equation*}
 \frac{\zeta^{n_{jk}(\Im \sigma_j-\Im\sigma_k)/m}}{\zeta^{\im n_{jk}\Re\sigma_k/m}}.
\end{equation*}
We thus see that we may also organize the series expansion of $\pi_{\bkappa^{-1}_{\zeta^{1/m}D},K}$ in the theorem as
\begin{equation*}
\pi_{\bkappa^{-1}_{\zeta^{1/m}D},K} = \sum_{\vartheta\in \mathfrak S_\R} \frac{\zeta^{-\im \vartheta/m}\tilde p_\vartheta(\zeta^{\im \Re\sigma_1/m},\dotsc,\zeta^{\im \Re\sigma_N/m},m^{-1} \log \zeta)}{\tilde q_{\vartheta}(\zeta^{\im \Re\sigma_1/m},\dotsc,\zeta^{\im \Re\sigma_N/m},m^{-1}\log\zeta)}
\end{equation*}
where $\mathfrak S_\R\subset \R$ is the additive semigroup generated by
\begin{equation*}
\set{\Im \sigma-\Im\sigma':\sigma,\sigma'\in \Sigma,\ \Im\sigma'\leq \Im\sigma'}
\end{equation*}
and $\tilde p_\vartheta$, $\tilde q_\vartheta$ are still polynomials.

\begin{remark}\label{RemarkAlignedB-Spec}
If $\Sigma$ lies on a line $\Re\sigma =c_0$, then $-\im\mathfrak S\subset \overline \R_--\im c_0$. Also in this case, the coefficients of the exponents in \eqref{TheVkAgain} can be assumed to have vanishing imaginary part (see Remark~\ref {RemarkAlignedSpec}). Assuming this, the coefficients of the exponents in \eqref{alpha0andtilde} are real, in particular $\det \alpha^{(0)}(t)$ is just a polynomial in $1/t$, the coefficients $r_\vartheta$ in the expansion \eqref{ProtoExpansion} can be written as rational functions of $t$ only. Consequently, in the expansion of the projection in Theorem~\ref{AsymptOfProjB}, the powers $-\im \vartheta$ are real $\leq 0$ and the coefficients can be written as rational functions of $\log\zeta$.
\end{remark}

\section{Asymptotic structure of the resolvent}\label{sec:AymtpoticsResolvent}

For the analysis of $(A_{\Dom}-\lambda)^{-\ell}$ for $\ell \in \N$ sufficiently
large we make use of the representation \eqref{ResolventStructure} of the resolvent as
\begin{gather}
(A_{\Dom}-\lambda)^{-1} = B(\lambda) + G_{\Dom}(\lambda), \label{ResolventStructureRepeat}
\intertext{where $B(\lambda)$ is a parametrix of 
$(A_{\min}-\lambda)$ and}
G_{\Dom}(\lambda)= [1-B(\lambda)(A-\lambda)]F_{\Dom}(\lambda)^{-1}T(\lambda).\label{GDPieces}
\end{gather}
The starting point of our analysis is
$$
(A_{\Dom}-\lambda)^{-\ell} = \frac{1}{(\ell-1)!}\partial_{\lambda}^{\ell-1}
(A_{\Dom}-\lambda)^{-1} \textup{ for any } \ell \in \N.
$$
We are thus led to further analyze the asymptotic structure of the pieces involved
in the representation of the resolvent.
In \cite{GKM5a} we described in full generality the behavior of
$B(\lambda)$, $[1-B(\lambda)(A-\lambda)]$, and $T(\lambda)$, and we analyzed
$F_{\Dom}(\lambda)^{-1}$ in the special case that $\Dom$ is stationary.
In the case of a general domain $\Dom$, we now obtain as a consequence of
Theorem~\ref{AsymptOfProjB} the following result.

\begin{proposition}\label{FInvStructure2}
For $R > 0$ large enough we have
$$
F_{\Dom}(\lambda)^{-1} \in \bigl(S_{\r}^{0^+} \cap S^{0}\bigr)
(\Lambda_R;\Dom_{\wedge}/\Dom_{\wedge,\min},\Dom_{\max}/\Dom_{\min}).
$$
The components of $F_{\Dom}(\lambda)^{-1}$ have orders $\nu^+$ with $\nu \in \E$,
the semigroup defined in \eqref{SemigroupDef}, and their phases belong to the
set $\M$ defined in \eqref{PhasesMainThm}.
\end{proposition}
Here $S^{0}(\Lambda_R;\Dom_{\wedge}/\Dom_{\wedge,\min},\Dom_{\max}/\Dom_{\min})$
denotes the standard space of (anisotropic) operator valued symbols of order zero
on $\Lambda_R$ (see the appendix), where $\Dom_{\wedge}/\Dom_{\wedge,\min}$ carries the
trivial group action, and  $\Dom_{\max}/\Dom_{\min}$ is equipped with the group action
$\tilde{\kappa}_{\varrho} = \theta^{-1}\bkappa_\rho\theta$.
The symbol class
$S_{\r}^{0^+}(\Lambda_R;\Dom_{\wedge}/\Dom_{\wedge,\min},\Dom_{\max}/\Dom_{\min})$
is discussed in the appendix (see Definition~\ref{symboldefinition}).
Recall that $\Lambda_R = \{\lambda \in \Lambda : |\lambda| \geq R\}$.

\begin{proof}[Proof of Proposition~\ref{FInvStructure2}]
We follow the line of reasoning of \cite[Propositions 5.10 and 5.17]{GKM5a}.
The crucial point is that we now know from Theorem~\ref{AsymptOfProjB}
and Proposition~\ref{FwedgeInvSymbol} that $F_{\wedge,\Dom_{\wedge}}(\lambda)^{-1}$
belongs to the symbol class
$$
\bigl(S_{\r}^{0^+} \cap S^{0}\bigr)
(\Lambda_R;\Dom_{\wedge}/\Dom_{\wedge,\min},\Dom_{\wedge,\max}/\Dom_{\wedge,\min}),
$$
where the actions on $\Dom_{\wedge}/\Dom_{\wedge,\min}$ and
$\Dom_{\wedge,\max}/\Dom_{\wedge,\min}$ are, respectively, the trivial action as above
and $\bkappa_{\varrho}$. The components of $F_{\wedge,\Dom_{\wedge}}(\lambda)^{-1}$
have orders $\nu^+$ with $\nu \in \E$, and their phases belong to the set $\M$.
Consequently, $\Phi_0(\lambda) = \theta^{-1}F_{\wedge,\Dom_{\wedge}}(\lambda)^{-1}$
belongs to
$$
\bigl(S_{\r}^{0^+} \cap S^{0}\bigr)
(\Lambda_R;\Dom_{\wedge}/\Dom_{\wedge,\min},\Dom_{\max}/\Dom_{\min}),
$$
and we have the same statement about the orders and phases of its components.

Phrased in the terminology of the present paper, we proved (see \cite[Proposition 5.10]{GKM5a}) that the operator family
$$
F(\lambda) = [T(\lambda)(A-\lambda)] : \Dom_{\max}/\Dom_{\min} \to 
\Dom_{\wedge}/\Dom_{\wedge,\min}
$$
belongs to the symbol class
$$
\bigl(S_{\r}^{0^+} \cap S^{0}\bigr)
(\Lambda_R;\Dom_{\max}/\Dom_{\min},\Dom_{\wedge}/\Dom_{\wedge,\min}),
$$
and that
$$
F(\lambda)\Phi_0(\lambda) - 1 = R(\lambda) \in
S^{-1+\eps}(\Lambda_R;\Dom_{\wedge}/\Dom_{\wedge,\min},\Dom_{\wedge}/\Dom_{\wedge,\min})
$$
for any $\eps > 0$. More precisely, $F(\lambda)$ is an anisotropic log-polyhomogeneous
operator valued symbol. We thus can infer further that in fact
$$
R(\lambda) \in
S_{\r}^{(-1)^+}(\Lambda_R;\Dom_{\wedge}/\Dom_{\wedge,\min},\Dom_{\wedge}/\Dom_{\wedge,\min}),
$$
and that the components of $R(\lambda)$ have orders $\nu^+$ with $\nu \in \E$,
$\nu \leq -1$, and phases belonging to the set $\M$.
The usual Neumann series argument then yields the existence of a symbol $R_1(\lambda) \in S_{\r}^{(-1)^+}(\Lambda_R;\Dom_{\wedge}/\Dom_{\wedge,\min},\Dom_{\wedge}/\Dom_{\wedge,\min})$
such that $F(\lambda)\Phi_0(\lambda)(1+R_1(\lambda)) = 1$ for $\lambda \in \Lambda_R$. 
Consequently $F_{\Dom}(\lambda)^{-1} = \Phi_0(\lambda)(1+R_1(\lambda))$ belongs to
$$
\bigl(S_{\r}^{0^+} \cap S^{0}\bigr)
(\Lambda_R;\Dom_{\wedge}/\Dom_{\wedge,\min},\Dom_{\max}/\Dom_{\min}),
$$
and its components have the structure that was claimed.
\end{proof}

With Proposition~\ref{FInvStructure2} and our results in \cite[Section 5]{GKM5a} at our disposal, we now obtain a general theorem about the asymptotics of the finite rank contribution $G_{\Dom}(\lambda)$ in the representation \eqref{ResolventStructureRepeat} of the resolvent. Before stating it we recall and rephrase the relevant results from \cite{GKM5a} about the other pieces involved in \eqref{GDPieces} using the terminology of the present paper.

Concerning $T(\lambda)$ we have (\cite[Proposition 5.5]{GKM5a}): 
\begin{enumerate}
\item [{\it i})]For any cut-off function $\omega \in C^{\infty}_c([0,1))$ the function
$T(\lambda)(1-\omega)$ is rapidly decreasing on $\Lambda$ taking values in
$\L(x^{-m/2}H^s_b,\Dom_{\wedge}/\Dom_{\wedge,\min})$, and
$$
t(\lambda) = T(\lambda)\omega \in S^{-m}(\Lambda;\K^{s,-m/2},\Dom_{\wedge}/\Dom_{\wedge,\min}).
$$
Here $\K^{s,-m/2}$ is equipped with the (normalized) dilation group action
$\kappa_{\varrho}$, and we give $\Dom_{\wedge}/\Dom_{\wedge,\min}$ again 
the trivial action.
\item [{\it ii})]The family $t(\lambda)$ admits a full asymptotic expansion into anisotropic homogeneous components. In particular, we have
$$
t(\lambda) \in S_{\r}^{(-m)^+}(\Lambda;\K^{s,-m/2},\Dom_{\wedge}/\Dom_{\wedge,\min}).
$$
\end{enumerate}
The spaces $\K^{s,-m/2}$ are weighted cone Sobolev spaces on $Y^{\wedge}$.
We discussed them in \cite[Section~2]{GKM2} and reviewed the definition in
\cite[Section~4]{GKM5a} (see also \cite{SchulzeNH}, where different
weight functions as $x \to \infty$ are considered).
Note that $\K^{0,-m/2} = x^{-m/2}L^2_b(Y^{\wedge};E)$.

Concerning $1-B(\lambda)(A-\lambda)$ we have, for arbitrary $\varphi \in C^{\infty}(M;\End(E))$ (\cite[Proposition 5.20]{GKM5a}):
\begin{enumerate}
\item [{\it iii})] The operator function
$P(\lambda) = \varphi[1-B(\lambda)(A-\lambda)]$
is a smooth function
\begin{equation*}
\Lambda_R \to  \L(\Dom_{\max}/\Dom_{\min},x^{-m/2}H^s_b)
\end{equation*}
which is defined for $R > 0$ large enough. Let $\omega \in C_c^{\infty}([0,1))$ be an arbitrary cut-off function. Then
$(1-\omega)P(\lambda)$ is rapidly decreasing on $\Lambda_R$, and
$$
p(\lambda) = \omega P(\lambda) \in
S^{0}(\Lambda_R;\Dom_{\max}/\Dom_{\min},\K^{s,-m/2}),
$$
where $\K^{s,-m/2}$ is equipped with the (normalized) dilation group action
$\kappa_{\varrho}$, and the quotient $\Dom_{\max}/\Dom_{\min}$ is equipped with
the group action $\tilde{\kappa}_{\varrho}$.

\item [{\it iv})] $p(\lambda)$ is an anisotropic log-polyhomogeneous operator valued symbol on $\Lambda_R$.
In particular,
$$
p(\lambda) = \omega P(\lambda) \in
S_{\r}^{0^+}(\Lambda_R;\Dom_{\max}/\Dom_{\min},\K^{s,-m/2}).
$$
\end{enumerate}

With $\M$ as in \eqref{PhasesMainThm} and $\E$ as in \eqref{SemigroupDef}
we have:

\begin{theorem}\label{GStructure2}
Let $\vp\in C^\infty(M;\textup{End}(E))$, and let
$\omega,\tilde\omega\in C_c^\infty([0,1))$ be arbitrary cut-off functions.
For $R > 0$ large enough the operator family $G_{\Dom}(\lambda)$ is defined
on $\Lambda_R$, and
$$
(1-\omega)\vp G_\Dom(\lambda), \;
\vp G_\Dom(\lambda)(1-\omega) \in
\S(\Lambda_R,\ell^1(x^{-m/2}H^s_b,x^{-m/2}H^t_b)).
$$
Moreover,
$$
\omega\vp G_\Dom(\lambda)\tilde \omega \in
\bigl(S_{\r}^{(-m)^+} \cap S^{-m}\bigr)(\Lambda_R;\K^{s,-m/2},\K^{t,-m/2}),
$$
where the spaces $\K^{s,-m/2}$ and $\K^{t,-m/2}$ are equipped with
the group action $\kappa_{\varrho}$.
In fact, $\omega\vp G_\Dom(\lambda)\tilde \omega$ takes values in the
trace class operators, and all statements about symbol estimates and
asymptotic expansions hold in trace class norms.
The components have orders $\nu^+$ with $\nu \in \E$, $\nu \leq -m$,
and their phases belong to $\M$.
\end{theorem}

\begin{corollary}\label{TraceGStructure2}
For $R > 0$ sufficiently large and $\vp\in C^\infty(M;\End(E))$, the operator family
$\varphi G_D(\lambda)$ is a smooth family of trace class operators in $x^{-m/2}L^2_b$
for $\lambda \in \Lambda_R$, and
$\Tr \bigl(\varphi G_D(\lambda)\bigr) \in
\bigl(S_{\r}^{(-m)^+} \cap S^{-m}\bigr)(\Lambda_R)$.
The components have orders $\nu^+$ with $\nu \in \E$, $\nu \leq -m$,
and their phases belong to the set $\M$.
\end{corollary}

Theorem~\ref{GStructure2} and Corollary~\ref{TraceGStructure2} follow at once from
the previous results about the pieces involved in the representation
\eqref{GDPieces} for $G_{\Dom}(\lambda)$ and the properties of the operator valued
symbol class discussed in the appendix.
In the statement of Corollary~\ref{TraceGStructure2} the scalar symbol
spaces are also anisotropic with anisotropy $m$. In particular, this means that
$\Tr \bigl(\varphi G_D(\lambda)\bigr) = O(|\lambda|^{-1})$ as $|\lambda| \to \infty$.

We are now in the position to prove the trace expansion claimed in
Theorem~\ref{ResolventTraceExpansion2}. To this end, we need the
following result (\cite[Theorem 4.4]{GKM5a}):

\begin{enumerate}
\item [{\it v})] Let $\varphi \in C^{\infty}(M;\End(E))$. If $m\ell >
n$, then
$\varphi\partial_{\lambda}^{\ell-1}B(\lambda)$ is a smooth
family of trace class operators in $x^{-m/2}L^2_b$, and the trace
$\Tr\bigl(\varphi\partial_{\lambda}^{\ell-1}B(\lambda)\bigr)$ is a
log-polyhomogeneous symbol on $\Lambda$. For large $\lambda$ we have
$$
\Tr\bigl(\varphi\partial_{\lambda}^{\ell-1}B(\lambda)\bigr) \sim
\sum\limits_{j=0}^{n-1}\alpha_{j}\lambda^{\frac{n-\ell m-j}{m}}
+ \alpha_{n}\log(\lambda) \lambda^{-\ell} + r(\lambda),
$$
where
$$
r(\lambda) \in \bigl(S_{\r}^{(-\ell m)^+} \cap S^{-\ell
m}\bigr)(\Lambda).
$$
\end{enumerate}

Now, combining {\it v}) with Corollary~\ref{TraceGStructure2}, we
finally obtain:

\begin{theorem}\label{MainThmTraceExpansion}
Let $\Lambda\subset\C$ be a closed sector. Assume that $A \in
x^{-m}\Diff^m_b(M;E)$, $m > 0$, with domain $\Dom\subset x^{-m/2}L^2_b$
satisfies the ray conditions \eqref{RayConditions}. Then $\Lambda$ is a
sector of minimal growth for $A_{\Dom}$, and for $m\ell > n$,
$(A_{\Dom}-\lambda)^{-\ell}$ is an analytic family of trace class
operators on $\Lambda_R$ for some $R>0$. Moreover, for $\varphi \in
C^{\infty}(M;\End(E))$, 
$$
\Tr\bigl(\varphi(A_{\Dom}-\lambda)^{-\ell}\bigr) \in
\bigl(S_{\r,\hol}^{(n-\ell m)^+} \cap S^{n-\ell m}\bigr)(\Lambda_R).
$$
The components have orders $\nu^+$ with $\nu \in \E$, $\nu \leq n-\ell m$, where $\E$ is the semigroup defined in \eqref{SemigroupDef}, and their phases belong to the set $\M$ defined in \eqref{PhasesMainThm}.
\end{theorem}
More precisely, we have the expansion
$$
\Tr \bigl(\varphi(A_{\Dom}-\lambda)^{-\ell}\bigr) \sim
\sum\limits_{j=0}^{n-1}\alpha_{j}\lambda^{\frac{n-\ell m-j}{m}}
+\alpha_{n}\log(\lambda) \lambda^{-\ell} + s_{\Dom}(\lambda)
$$
with constants $\alpha_j \in \C$ independent of the choice of domain $\Dom$, and a domain dependent remainder $s_{\Dom}(\lambda) \in \bigl(S_{\r,\hol}^{(-\ell m)^+} \cap S^{-\ell m}\bigr)(\Lambda_R)$.

If all elements of the set $\{\sigma \in \spec_b(A) : -m/2 < \Im\sigma < m/2\}$ are vertically aligned, then the coefficients $r_{\nu}$ in the expansion \eqref{sDExpansion} of $s_{\Dom}(\lambda)$ are rational functions of $\log\lambda$ only. This is because, in this case, the series representation of the projection in Theorem~\ref{AsymptOfProjB} contains only real powers of $\zeta$ and rational functions of $\log\zeta$, see Remark~\ref{RemarkAlignedB-Spec}. This simplifies the structure of $F_{\wedge,\Dom_{\wedge}}(\lambda)^{-1}$ according to Section~\ref{sec:ModelResolvent}, and consequently the structure of $F_{\Dom}(\lambda)^{-1}$ (see the proof of Proposition~\ref{FInvStructure2}). As recalled in this section, the terms coming from $B(\lambda)$ and the other pieces in the representation \eqref{GDPieces} of $G_{\Dom}(\lambda)$ do not generate phases. 

If $\Dom$ is stationary, then the expansion \eqref{sDExpansion} of $s_{\Dom}(\lambda)$ is even simpler: the $r_\nu$ are just polynomials in $\log\lambda$, and the numbers $\nu$ are all integers. To see this recall that if $\Dom_\wedge$ is $\kappa$-invariant, then $F_{\wedge,\Dom_{\wedge}}(\lambda)^{-1}$ is homogeneous, see \eqref{FwedgeInvHomogeneity}, so it belongs to the class 
$$
S^{(0)}(\Lambda_R;\Dom_{\wedge}/\Dom_{\wedge,\min},\Dom_{\wedge,\max}/\Dom_{\wedge,\min}) \!\subset\! \bigl(S_{\r}^{0^+} \cap S^{0}\bigr)
(\Lambda_R;\Dom_{\wedge}/\Dom_{\wedge,\min},\Dom_{\wedge,\max}/\Dom_{\wedge,\min}).
$$
Consequently, by the proof of Proposition~\ref{FInvStructure2}, $F_{\Dom}(\lambda)^{-1}$ is $\log$-polyhomogeneous. This property propagates throughout the rest of the results in this section and gives the structure of $s_\Dom(\lambda)$ just asserted.

\begin{appendix}
\section{A class of symbols}\label{AppendixSymbolSpaces}

Let $\Lambda \subset \C$ be a closed sector.
Let $E$ and $\tilde{E}$ be Hilbert spaces equipped with strongly continuous
group actions $\kappa_{\varrho}$ and $\tilde{\kappa}_{\varrho}$, $\varrho > 0$,
respectively. Recall that the space $S^{\nu}(\Lambda;E,\tilde{E})$ of anisotropic
operator valued symbols on the sector $\Lambda$ of order $\nu \in \R$ is defined as
the space of all $a \in C^{\infty}(\Lambda,\L(E,\tilde{E}))$ such that for all
$\alpha,\beta \in \N_0$
$$
\|\tilde{\kappa}^{-1}_{|\lambda|^{1/m}}
\partial^{\alpha}_{\lambda}\partial^{\beta}_{\bar{\lambda}}a(\lambda)
\kappa_{|\lambda|^{1/m}}\|_{\L(E,\tilde{E})} = \Oh(|\lambda|^{\nu/m - \alpha - \beta})
\textup{ as } |\lambda| \to \infty \textup{ in } \Lambda.
$$
By $S^{(\nu)}(\Lambda;E,\tilde{E})$ we denote the space of anisotropic homogeneous
functions of degree $\nu \in \R$, i.e., all $a \in
C^{\infty}(\Lambda\setminus\{0\},\L(E,\tilde{E}))$ such that
$$
a(\varrho^m\lambda) = \varrho^{\nu}
\tilde{\kappa}_{\varrho}a(\lambda)\kappa_{\varrho}^{-1} \textup{ for } \varrho > 0
\textup{ and } \lambda \in \Lambda\setminus\{0\}.
$$
Clearly $\chi(\lambda)S^{(\nu)}(\Lambda;E,\tilde{E}) \subset 
S^{\nu}(\Lambda;E,\tilde{E})$ with the obvious meaning of notation,
where $\chi \in C^{\infty}(\R^2)$ is any excision function of the origin.
When $E=\tilde{E}=\C$ equipped with the trivial group action the spaces are
dropped from the notation.

Such symbol classes were introduced by Schulze in his theory of
pseudodifferential operators on manifolds with singularities, see \cite{SchulzeNH}.
In particular, classical symbols, i.e. symbols that admit asymptotic expansions into
homogeneous components, play a prominent role, and we have used such symbols
in \cite{GKM2} for the construction of a parameter-dependent parametrix
$B(\lambda)$ to $A_{\min}-\lambda$.
As we see in the present paper, the structure of resolvents $(A_{\Dom}-\lambda)^{-1}$
for general domains $\Dom$ is rather involved, and classical symbols do not
suffice to describe that structure.
We are therefore led to introduce a new class of (anisotropic) operator valued symbols
that admit certain expansions of more general kind. As turns out, this class
occurs naturally and is well adapted to describe the structure of
resolvents in the general case.

\medskip

\noindent
Recall that $V[z_1,\ldots,z_M]$ denotes the space of polynomials in the variables
$z_j$, $j=1\ldots,M$, with coefficients in $V$ for any vector space $V$. We shall make
use of this in particular for $V = \C$ and $V = S^{(0)}(\Lambda;E,\tilde{E})$.
In what follows, all holomorphic powers and logarithms on $\open\Lambda$
are defined using a holomorphic branch of the logarithm with cut
$\Gamma \not\subset \Lambda$.

\begin{definition}\label{componentsdef}
Let $\nu \in \R$. We define $S_{\r}^{(\nu^+)}(\Lambda;E,\tilde{E})$ as the space of
all functions $s(\lambda)$ of the following form:

There exist polynomials $p \in S^{(0)}(\Lambda;E,\tilde{E})[z_1,\ldots,z_{N+1}]$
and $q \in \C[z_1,\ldots,z_{N+1}]$ in $N+1$ variables, $N=N(s) \in \N_0$,
and real numbers $\mu_{k}=\mu_k(s)$, $k = 1,\ldots,N$, such that the following holds:
\begin{enumerate}[(a)]
\item $|q(\lambda^{i\mu_{1}},\ldots,\lambda^{i\mu_{N}},\log\lambda)|
\geq c > 0$ for $\lambda \in \Lambda$ with $|\lambda|$ sufficiently large;
\item $s(\lambda) = r(\lambda)\lambda^{\nu/m}$, where
\begin{equation}\label{Rationalcoeff}
r(\lambda) =
\frac{p(\lambda^{i\mu_{1}},\ldots,\lambda^{i\mu_{N}},\log\lambda)}{q(\lambda^{i\mu_{1}},\ldots,\lambda^{i\mu_{N}},\log\lambda)}.
\end{equation}
To clarify the notation, we note that
$$
p(\lambda^{i\mu_{1}},\ldots,\lambda^{i\mu_{N}},\log\lambda) =
\sum_{|\alpha|+k \leq M}a_{\alpha,k}(\lambda)\lambda^{i\mu_1\alpha_1}\cdots
\lambda^{i\mu_N\alpha_N}\log^k\lambda
$$
as a function $\Lambda\setminus\{0\} \to \L(E,\tilde{E})$ with certain
$a_{\alpha,k}(\lambda) \in S^{(0)}(\Lambda;E,\tilde{E})$.
\end{enumerate}
We call the $\mu_k$ the phases and $\nu^+$ the order of $s(\lambda)$.

Every $s(\lambda) \in S_{\r}^{(\nu^+)}(\Lambda;E,\tilde{E})$ is an operator function
defined everywhere on $\Lambda$ except at $\lambda = 0$ and the zero set of
$q(\lambda^{i\mu_{1}},\ldots,\lambda^{i\mu_{N}},\log\lambda)$.
The latter is a discrete subset of $\Lambda\setminus\{0\}$, and it is finite
outside any neighborhood of zero in view of property (a).
\end{definition}

\begin{proposition}\label{componentprop}
\begin{enumerate}
\item $S_{\r}^{(\nu^+)}(\Lambda;E,\tilde{E})$ is a vector space.
\item Let $\hat{E}$ be a third Hilbert space with group action $\hat{\kappa}_{\varrho}$,
$\varrho > 0$. Composition of operator functions induces a map
$$
S_{\r}^{(\nu_1^+)}(\Lambda;\tilde{E},\hat{E})\times
S_{\r}^{(\nu_2^+)}(\Lambda;E,\tilde{E}) \to
S_{\r}^{((\nu_1+\nu_2)^+)}(\Lambda;E,\hat{E}).
$$
\item For $\alpha,\beta \in \N_0$ we have
$$
\partial_{\lambda}^{\alpha}\partial_{\bar{\lambda}}^{\beta} :
S_{\r}^{(\nu^+)}(\Lambda;E,\tilde{E}) \to
S_{\r}^{((\nu-m\alpha-m\beta)^+)}(\Lambda;E,\tilde{E}).
$$
\item Let $s(\lambda) \in S_{\r}^{(\nu^+)}(\Lambda;E,\tilde{E})$. Then
$$
\chi(\lambda)s(\lambda) \in S^{\nu+\eps}(\Lambda;E,\tilde{E})
$$
for any $\eps > 0$ and any excision function $\chi \in C^{\infty}(\R^2)$ of
the set where $s(\lambda)$ is undefined.
\item Let $s(\lambda) \in S_{\r}^{(\nu^+)}(\Lambda;E,\tilde{E})$ and assume
that
$$
\|\tilde{\kappa}_{|\lambda|^{1/m}}^{-1}s(\lambda)
\kappa_{|\lambda|^{1/m}}\|_{\L(E,\tilde{E})} = \Oh(|\lambda|^{\nu/m-\eps})
$$
as $|\lambda| \to \infty$ for some $\eps > 0$. Then $s(\lambda) \equiv 0$
on $\Lambda$.

In particular,
$S_{\r}^{(\nu_1^+)}(\Lambda;E,\tilde{E}) \cap S_{\r}^{(\nu_2^+)}(\Lambda;E,\tilde{E})
= \{0\}$ whenever $\nu_1 \neq \nu_2$.
\end{enumerate}
\end{proposition}
\begin{proof}
(1) and (2) are obvious. For (3) note that
$$
\partial_{\lambda}^{\alpha}\partial_{\bar{\lambda}}^{\beta} :
S^{(\nu_0)}(\Lambda;E,\tilde{E}) \to S^{(\nu_0-m\alpha-m\beta)}(\Lambda;E,\tilde{E})
$$
for any $\nu_0$. Consequently,
$\partial_{\lambda}^{\alpha}\partial_{\bar{\lambda}}^{\beta}$ acts in the spaces
\begin{align*}
S^{(\nu_0)}(\Lambda;E,\tilde{E})[\lambda^{i\mu_1},\ldots,\lambda^{i\mu_N},\log\lambda]
&\to
S^{(\nu_0-m\alpha-m\beta)}(\Lambda;E,\tilde{E})[\lambda^{i\mu_1},\ldots,\lambda^{i\mu_N},\log\lambda] \\
\C[\lambda^{i\mu_1},\ldots,\lambda^{i\mu_N},\log\lambda] &\to
S^{(-m\alpha-m\beta)}(\Lambda)[\lambda^{i\mu_1},\ldots,\lambda^{i\mu_N},\log\lambda]
\end{align*}
with the obvious meaning of notation (the latter is a special case of the
former in view of $\C \subset S^{(0)}(\Lambda)$).
(3) is an immediate consequence of these observations.

(4) follows at once in view of property (a) in Definition~\ref{componentsdef} (and
using (3) to estimate higher derivatives).
Note also that, for large $\lambda$, the numerator in \eqref{Rationalcoeff} can be
regarded as a polynomial in $\log\lambda$ of operator valued symbols of order zero.

In the proof of (5) we may without loss of generality assume that $\nu = 0$, so
$s(\lambda)$ is of the form \eqref{Rationalcoeff}. Since
$|q(\lambda^{i\mu_{1}},\ldots,\lambda^{i\mu_{N}},\log\lambda)| = \Oh(\log^M|\lambda|)$
as $|\lambda| \to \infty$ we see that it is sufficient to consider the case
$q \equiv 1$, so
$s(\lambda) = p(\lambda^{i\mu_{1}},\ldots,\lambda^{i\mu_{N}},\log\lambda)$.
For this case we will prove that if
$$
\|\tilde{\kappa}_{|\lambda|^{1/m}}^{-1}s(\lambda)
\kappa_{|\lambda|^{1/m}}\|_{\L(E,\tilde{E})} \to 0
$$
as $|\lambda| \to \infty$, then $s(\lambda) \equiv 0$ on $\Lambda$.
For this proof we can without loss of generality further assume that
$s(\lambda)$ contains no logarithmic terms, so we have $s(\lambda) =
p(\lambda^{i\mu_{1}},\ldots,\lambda^{i\mu_{N}})$. Moreover, we can assume that
the numbers $\mu_1,\ldots,\mu_N \in \R$ are independent over the rationals,
for if this is not the case we can choose rationally independent numbers
$\tilde{\mu}_1,\ldots, \tilde{\mu}_K \in \R$ such that
$\mu_j = \sum_{k=1}^K z_{jk}\tilde{\mu}_k$ with coefficients $z_{jk} \in \Z$, and so
$$
\lambda^{i\mu_j} = \prod_{k=1}^K\bigl(\lambda^{i\tilde{\mu}_k}\bigr)^{z_{jk}}
$$
for every $j = 1,\ldots,N$. Consequently, there are numbers $N_j \in \N$,
$j=1,\ldots,K$, and a polynomial $\tilde{p} \in
S^{(0)}(\Lambda;E,\tilde{E})[z_1,\ldots,z_K]$ such that
$$
\lambda^{i\tilde{\mu}_1N_1}\cdots\lambda^{i\tilde{\mu}_KN_K}
p(\lambda^{i\mu_{1}},\ldots,\lambda^{i\mu_{N}}) =
\tilde{p}(\lambda^{i\tilde{\mu}_{1}},\ldots,\lambda^{i\tilde{\mu}_{K}}),
$$
and both assertion and assumption are valid for $p$ if and only if they hold
for $\tilde{p}$. So we can indeed assume that the numbers $\mu_j$, $j = 1,\ldots,N$,
are independent over the rationals.

Now let $\lambda_0 \in \Lambda$ be arbitrary with $|\lambda_0|=1$, and consider the
function $f : (0,\infty) \to \L(E,\tilde{E})$ defined by
$$
f(\varrho) = \tilde{\kappa}_{\varrho}^{-1}p(\varrho^{im\mu_1}\lambda_0^{i\mu_1},
\ldots,\varrho^{im\mu_N}\lambda_0^{i\mu_N})\kappa_{\varrho}.
$$
This function is of the form
$$
f(\varrho) = \sum_{|\alpha| \leq M}a_{\alpha}(\varrho^{i\mu_1})^{\alpha_1}\cdots
(\varrho^{i\mu_N})^{\alpha_N}
$$
for certain $a_{\alpha} \in \L(E,\tilde{E})$, and by assumption
$\|f(\varrho)\|_{\L(E,\tilde{E})} \to 0$ as $\varrho \to \infty$.
Let $p_0(z) = \sum_{|\alpha| \leq M}a_{\alpha}z^{\alpha}$, $z = (z_1,\ldots,z_N) \in
\C^N$, and consider the curve
$$
\varrho \mapsto (\varrho^{i\mu_1},\ldots,\varrho^{i\mu_N}) \in
{\mathbb S}^1\times\ldots\times{\mathbb S}^1
$$
on the $N$-torus. The image of this curve for $\varrho > \varrho_0$ is a dense
subset of the $N$-torus, where $\varrho_0 > 0$ can be chosen arbitrarily, because the
$\mu_j$ are independent over the rationals.
The function $f$ is merely the operator polynomial $p_0(z)$ restricted to that curve.
Since $f(\varrho) \to 0$ as $\varrho \to \infty$, this implies that for any $\eps > 0$
we have $\|p_0(z)\| < \eps$ for all $z$ in a dense subset of the $N$-torus. This
shows that $p_0(z)$ is the zero polynomial, and so the function
$f(\varrho) = 0$ for all $\varrho > 0$.

Consequently, the function $p(\lambda^{i\mu_{1}},\ldots,\lambda^{i\mu_{N}})$
vanishes along the ray through $\lambda_0$, and because $\lambda_0$ was arbitrary
the proof is complete.
\end{proof}

\begin{definition}\label{symboldefinition}
For $\nu \in \R$ define $S^{\nu^+}_{\r}(\Lambda;E,\tilde{E})$ as the space of all
operator valued symbols $a(\lambda)$ that admit an asymptotic expansion
\begin{equation}\label{AsymptoticDefSymbol}
a(\lambda) \sim \sum_{j=0}^{\infty}\chi_j(\lambda)s_j(\lambda),
\end{equation}
where $s_j(\lambda) \in S^{(\nu_j^+)}_{\r}(\Lambda;E,\tilde{E})$,
$\nu=\nu_0 > \nu_1 > \ldots$ and $\nu_j \to -\infty$ as $j \to \infty$, and
$\chi_j(\lambda)$ is a suitable excision function of the set where
$s_j(\lambda)$ is undefined.

We call $s_j(\lambda)$ the component of order $\nu_j^+$ of $a(\lambda)$.
The components are uniquely determined by the symbol $a(\lambda)$
(see Proposition~\ref{symbolprop}).
\end{definition}

Familiar symbol classes like classical (polyhomogeneous) symbols,
symbols that admit asymptotic expansions into homogeneous components
of complex degrees, or log-polyhomogeneous symbols are all particular
cases of the class defined in Definition~\ref{symboldefinition}.
In particular, the denominators $q$ in \eqref{Rationalcoeff}
are equal to one in all those cases.

Of particular interest in the context of this paper are symbols $a(\lambda)$ with the
property that all components $s_j(\lambda)$ have orders $\nu_j^+$
with $\nu_j \in \E$, the semigroup defined in \eqref{SemigroupDef}, and phases
in the set $\M$ defined in \eqref{PhasesMainThm}.

\begin{proposition}\label{symbolprop}
\begin{enumerate}
\item $S_{\r}^{\nu^+}(\Lambda;E,\tilde{E})$ is a vector space. For any
$\eps > 0$ we have the inclusion
$S_{\r}^{\nu^+}(\Lambda;E,\tilde{E}) \subset S^{\nu+\eps}(\Lambda;E,\tilde{E})$.
\item Let $a(\lambda) \in S_{\r}^{\nu^+}(\Lambda;E,\tilde{E})$. The components
$s_j(\lambda)$ in \eqref{AsymptoticDefSymbol} are uniquely determined by $a(\lambda)$.
\item Let $\hat{E}$ be a third Hilbert space with group action $\hat{\kappa}_{\varrho}$,
$\varrho > 0$. Composition of operator functions induces a map
$$
S_{\r}^{\nu_1^+}(\Lambda;\tilde{E},\hat{E})\times
S_{\r}^{\nu_2^+}(\Lambda;E,\tilde{E}) \to
S_{\r}^{(\nu_1+\nu_2)^+}(\Lambda;E,\hat{E}).
$$
The components of the composition of two symbols are obtained by formally multiplying
the asymptotic expansions \eqref{AsymptoticDefSymbol} of the factors.
\item For $\alpha,\beta \in \N_0$ we have
$$
\partial_{\lambda}^{\alpha}\partial_{\bar{\lambda}}^{\beta} :
S_{\r}^{\nu^+}(\Lambda;E,\tilde{E}) \to
S_{\r}^{(\nu-m\alpha-m\beta)^+}(\Lambda;E,\tilde{E}).
$$
If $s_j(\lambda)$ are the components of $a(\lambda) \in
S_{\r}^{\nu^+}(\Lambda;E,\tilde{E})$, then
$\partial_{\lambda}^{\alpha}\partial_{\bar{\lambda}}^{\beta}s_j(\lambda)$ are
the components of
$\partial_{\lambda}^{\alpha}\partial_{\bar{\lambda}}^{\beta}a(\lambda)$.
\item Let $a_j(\lambda) \in S_{\r}^{\nu_j^+}(\Lambda;E,\tilde{E})$, $\nu_j \to -\infty$
as $j \to \infty$, and let $\bar{\nu} = \max\nu_j$.
Let $a(\lambda)$ be an operator valued symbol such that
$a(\lambda) \sim \sum_{j=0}^{\infty}a_j(\lambda)$.

Then $a(\lambda) \in
S_{\r}^{{\bar{\nu}}^+}(\Lambda;E,\tilde{E})$, and the component of $a(\lambda)$
of order $M^+$ is obtained by adding the components of that order
of the $a_j(\lambda)$. This is a finite sum for each $M \leq \bar{\nu}$ and
will yield a nontrivial result for at most countably many values of $M$ that form a
sequence tending to $-\infty$.
\end{enumerate}
\end{proposition}
\begin{proof}
Everything follows from Proposition~\ref{componentprop} and standard arguments.
Because of its importance we will, however, prove (2):

To this end, assume that $0 \sim \sum_{j=0}^{\infty}\chi_j(\lambda)s_j(\lambda)$
with $s_j(\lambda) \in S^{(\nu_j^+)}_{\r}(\Lambda;E,\tilde{E})$,
$\nu_j > \nu_{j+1} \to -\infty$ as $j \to \infty$. We need to prove that all
$s_j(\lambda)$ are zero. Because
$$
\chi_0(\lambda)s_0(\lambda) \sim -\sum_{j=1}^{\infty}\chi_j(\lambda)s_j(\lambda)
$$
we see that $\chi_0(\lambda)s_0(\lambda) \in S^{\nu_1+\eps}(\Lambda;E,\tilde{E})$
for every $\eps > 0$. Choose $\eps > 0$ such that $\nu_1 + \eps < \nu_0$. Then
$\|\tilde{\kappa}^{-1}_{|\lambda|^{1/m}}\chi_0(\lambda)s_0(\lambda)
\kappa_{|\lambda|^{1/m}}\|_{\L(E,\tilde{E})} = \Oh(|\lambda|^{(\nu_1+\eps)/m})$
as $|\lambda| \to \infty$, and by Proposition~\ref{componentprop}(5) we
obtain that $s_0(\lambda) \equiv 0$ on $\Lambda$. Consequently all $s_j(\lambda)$
are zero by induction, and (2) is proved.
\end{proof}

By $S_{\r,\hol}^{\nu^+}(\Lambda;E,\tilde{E})$ we denote the class of
symbols $a(\lambda) \in S_{\r}^{\nu^+}(\Lambda;E,\tilde{E})$ that
are holomorphic in $\open\Lambda$.
Let $s_j(\lambda)$ be the components of $a(\lambda) \in
S_{\r,\hol}^{\nu^+}(\Lambda;E,\tilde{E})$.
By Proposition~\ref{symbolprop}, $\partial_{\bar{\lambda}}s_j(\lambda)$ are
the components of $\partial_{\bar{\lambda}}a(\lambda) \equiv 0$,
and consequently all components $s_j(\lambda)$ are holomorphic.

In the case of holomorphic scalar symbols (or, more generally, holomorphic
operator valued symbols with trivial group actions), we can improve the
description of the components as follows.

\begin{proposition}\label{ScalarSymbolStructure}
Let $a(\lambda) \in S_{\r,\hol}^{\nu^+}(\Lambda)$,
$a(\lambda) \sim \sum_{j=0}^{\infty}\chi_j(\lambda)s_j(\lambda)$
with components $s_j(\lambda)$ of order $\nu_j^+$.

For every $j \in \N_0$ there exist polynomials $p_j,q_j \in \C[z_1,\ldots,z_{N_j+1}]$
in $N_j+1$ variables with constant coefficients, $N_j \in \N_0$,
and real numbers $\mu_{jk}$, $k = 1,\ldots,N_j$, such that the following holds:
\begin{enumerate}[(a)]
\item $|q_j(\lambda^{i\mu_{j1}},\ldots,\lambda^{i\mu_{jN_j}},\log\lambda)|
\geq c_j > 0$ for $\lambda \in \Lambda$ with $|\lambda|$ sufficiently large;
\item $s_j(\lambda) = r_j(\lambda^{i\mu_{j1}},\ldots,\lambda^{i\mu_{jN_j}},\log\lambda)
\lambda^{\nu_j/m}$, where $r_j = p_j/q_j$.
\end{enumerate}
\end{proposition}
\begin{proof}
We already know that the components $s_j(\lambda)$ are holomorphic.
We just need to show that in this case the numerator polynomials $p$ in
Definition~\ref{componentsdef} can be chosen to have constant
coefficients rather than homogeneous coefficient functions.
This, however, follows from Lemma~\ref{UniqAnalCont} below.
\end{proof}

\begin{lemma}\label{UniqAnalCont}
Let $f_1(\lambda),\ldots,f_M(\lambda)$ be holomorphic functions on
$\Lambda\setminus\{0\}$, and let $p \in S^{(0)}(\Lambda)[z_1,\ldots,z_M]$.
Assume that the function $p(f_1(\lambda),\ldots,f_M(\lambda))$ is holomorphic
on $\open\Lambda$, except possibly on a discrete set.

Then there is a polynomial $p_0 \in \C[z_1,\ldots,z_M]$ with constant coefficients
such that
$$
p(f_1(\lambda),\ldots,f_M(\lambda)) = p_0(f_1(\lambda),\ldots,f_M(\lambda))
$$
as functions on $\Lambda\setminus\{0\}$.
\end{lemma}
\begin{proof}
Since all singularities are removable, we know that
$p(f_1(\lambda),\ldots,f_M(\lambda))$ is holomorphic everywhere on $\open\Lambda$.
We have
\begin{align*}
p(f_1(\lambda),\ldots,f_M(\lambda)) &= \sum\limits_{|\alpha| \leq D}
a_{\alpha}(\lambda/|\lambda|)f_1(\lambda)^{\alpha_1}\cdots f_M(\lambda)^{\alpha_M}. \\
\intertext{Let $\lambda_0 \in \open\Lambda$. Define}
p_0(z_1,\ldots,z_M) &= \sum\limits_{|\alpha| \leq D}a_{\alpha}(\lambda_0/|\lambda_0|)
z_1^{\alpha_1}\cdots z_M^{\alpha_M}
\end{align*}
Then clearly
$$
p(f_1(\lambda),\ldots,f_M(\lambda)) = p_0(f_1(\lambda),\ldots,f_M(\lambda))
$$
on the ray through $\lambda_0$. By uniqueness of analytic continuation this
equality necessarily holds everywhere on $\open\Lambda$, and by continuity
then also on $\Lambda\setminus\{0\}$.
\end{proof}
\end{appendix}



\begin{thebibliography}{10}

\bibitem{BrSe85}
J.~Br\"uning and R.~Seeley, \emph{Regular singular asymptotics},
Adv. in Math.  \textbf{58} (1985), 133--148.

\bibitem{BrSe87}
\bysame, \emph{The resolvent expansion for second order regular singular operators}, J. Funct. Anal. \textbf{73} (2) (1987) 369--429. 

\bibitem{BrSe91}
\bysame, \emph{The expansion of the resolvent near a singular stratum of conical type}, J. Funct. Anal. \textbf{95} (1991), no. 2, 255--290.

\bibitem{Ca83}
C.~Callias, \emph{The heat equation with singular coefficients I. Operators of the form $-d^{2}/dx^{2}+k/x^{2}$ in dimension 1}, Comm. Math. Phys. \textbf{88} (3) (1983) 357--385.

\bibitem{Ch79}
J.~Cheeger, \emph{On the spectral geometry of spaces with cone-like singularities}, Proc. Nat. Acad. Sci. USA \textbf{76} (1979), 2103--2106.

\bibitem{FMP04}
H.~Falomir, M.A.~Muschietti, P.A.G.~Pisani, \emph{On the resolvent and spectral functions of a second order differential operator with a regular singularity}, J. Math. Phys. \textbf{45} (12) (2004) 4560--4577.

\bibitem{FMPSeeley}
H.~Falomir, M.A.~Muschietti, P.A.G.~Pisani, and R.~Seeley,  
\emph{Unusual poles of the $\zeta$-functions for some regular singular differential operators}, J. Phys. A \textbf{36} (2003), no. 39, 9991--10010.

\bibitem{FPW}
H.~Falomir, P.A.G.~Pisani, and A.~Wipf, \emph{Pole structure of the Hamiltonian $\zeta$-function for a singular potential}, J. Phys. A \textbf{35} (2002), 5427--5444.

\bibitem{GilThesis}
J.~Gil, \emph{Heat trace asymptotics for cone differential operators}, Ph.D. thesis,
Universit{\"a}t Potsdam, 1998.

\bibitem{GiHeat01}
\bysame, \emph{Full asymptotic expansion of the heat trace for non-self-adjoint elliptic cone operators}, Math. Nachr. \textbf{250} (2003), 25--57.

\bibitem{GKM1}
J.~Gil, T.~Krainer, and G.~Mendoza, \emph{Geometry and spectra of 
closed extensions of elliptic cone operators}, 
Canad. J. Math. \textbf{59} (2007), no. 4, 742--794.

\bibitem{GKM2}
\bysame, \emph{Resolvents of elliptic cone operators}, J. Funct. Anal. \textbf{241} (2006), 
no. 1, 1--55.

\bibitem{GKM3}
\bysame, \emph{On rays of minimal growth for elliptic cone operators}, Oper.~Theory Adv.~Appl. \textbf{172} (2007), 33--50.

\bibitem{GKM5a}
\bysame, \emph{Trace expansions for elliptic cone operators with stationary domains}, submitted.

\bibitem{GiLo08}
J.~Gil and P.~Loya, \emph{Resolvents of cone pseudodifferential operators, asymptotic expansions and applications}, Math. Z. \textbf{259} (2008), no. 1, 65--95. 

\bibitem{GiMe01}
J.~Gil and G.~Mendoza, \emph{Adjoints of elliptic cone operators},
Amer. J. Math. \textbf{125} (2003) 2, 357--408.

\bibitem{Gilkey}
P.~Gilkey, \emph{Invariance theory, the heat equation, and the {A}tiyah-{S}inger index theorem}, CRC Press, Boca Raton, Ann Arbor, 1996, second edition.

\bibitem{GruSee95}
G.~Grubb and R.~Seeley, \emph{Weakly parametric pseudodifferential operators and Atiyah--Patodi--Singer boundary problems}, Invent. Math. \textbf{121}, 481--529 (1995). 

\bibitem{KLP06}
K.~Kirsten, P.~Loya, and J.~Park, \emph{The very unusual properties of the resolvent, heat kernel, and zeta function for the operator $-d\sp 2/dr\sp 2-1/(4r\sp 2)$}, 
J. Math. Phys. \textbf{47} (2006), no. 4, 043506, 27 pp.

\bibitem{KLP08a}
\bysame, \emph{Functional determinants for general self-adjoint extensions of Laplace-type operators resulting from the generalized cone}, Manuscripta Math. \textbf{125} (2008), no. 1, 95--126.

\bibitem{KLP08b}
\bysame, \emph{Exotic expansions and pathological properties of $\zeta$-functions on conic manifolds}, J. Geom. Anal. \textbf{18} (2008), no. 3, 835--888.

\bibitem{Le97} 
M.~Lesch, \emph{Operators of {F}uchs type, conical singularities, and asymptotic methods}, Teubner-Texte zur Math. vol 136, B.G. Teubner, Stuttgart, Leipzig, 1997.

\bibitem{LoRes01}
P.~Loya, \emph{On the resolvent of differential operators on conic manifolds}, 
Comm. Anal. Geom. \textbf{10} (2002), no.~5, 877--934.

\bibitem{LoRes201}
\bysame, \emph{Parameter-dependent operators and resolvent expansions on conic manifolds}, Illinois J. Math. \textbf{46} (2002), no.~4, 1035--1059.

\bibitem{LMP07}
P.~Loya, P.~McDonald, and J.~Park, \emph{Zeta regularized determinants for conic manifolds},  J. Funct. Anal. \textbf{242} (2007), no. 1, 195--229.

\bibitem{RBM2} 
R.~Melrose, \emph{The Atiyah-Patodi-Singer index theorem}, Research Notes in Mathematics, A~K~Peters, Ltd., Wellesley, MA, 1993.

\bibitem{Mooers}
E.~Mooers, \emph{Heat kernel asymptotics on manifolds with conic singularities}, J. Anal. Math. \textbf{78} (1999) 1--36.

\bibitem{SchulzeNH} 
B.--W. Schulze, \emph{Pseudo-differential operators on manifolds with singularities}, Studies in Mathematics and its Applications, 24. North-Holland Publishing Co., Amsterdam, 1991.

\bibitem{Seeley}
R.~Seeley, \emph{Complex powers of an elliptic operator}, Singular Integrals, AMS Proc. Symp. Pure Math. X, 1966, Amer. Math. Soc., Providence, 1967, pp.~288--307.

\end{thebibliography}
\end{document}